\documentclass[12pt,reqno]{amsart}

\usepackage{amssymb}
\usepackage{bm}
\usepackage{graphicx}
\usepackage{psfrag}
\usepackage[usenames,dvipsnames]{xcolor}
\usepackage{enumerate}
\usepackage{subcaption}
\usepackage{url}

\usepackage{algpseudocode}

\usepackage{mathtools}

\usepackage{xy}
\input xy
\xyoption{all}

\newcommand{\excise}[1]{}%{$\star$\textsc{#1}$\star$}

\newtheorem{thm}{Theorem}[section]
\newtheorem{lemma}[thm]{Lemma}

\newtheorem{cor}[thm]{Corollary}
\newtheorem{prop}[thm]{Proposition}
\newtheorem{conj}[thm]{Conjecture}

\theoremstyle{definition}

\newtheorem{example}[thm]{Example}
\newtheorem{remark}[thm]{Remark}
\newtheorem{definition}[thm]{Definition}

\newtheorem{question}[thm]{Question}
\newtheorem{prob}[thm]{Problem}

\numberwithin{equation}{section}

\newcommand{\ring}[1]{\ensuremath{\mathbb{#1}}}

\newcommand\RR{\ring{R}}

\newcommand\ZZ{\ring{Z}}

\def\ol#1{{\overline {#1}}}

\DeclareMathOperator\ann{ann} % Annihilator
\newcommand\R{{\mathbb R}}

\begin{document}%%%%%%%%%%%%%%%%%%%%%%%%%%%%%%%%%%%%%%%%%%%%%%%%%%%%%%%%
%%%%%%%%%%%%%%%%%%%%%%%%%%%%%%%%%%%%%%%%%%%%%%%%%%%%%%%%%%%%%%%%%%%%%%%%

\mbox{}
%\vspace{-2ex}%-1.1743pt}
\title[The geometry and combinatorics of discrete line segment hypergraphs]{The geometry and combinatorics of \\ discrete line segment hypergraphs}

\author[D.~Oliveros]{Deborah Oliveros}
\address{Instituto de Matem\'aticas\\Universidad Nacional Aut\'onoma de M\'exico\\}
\email{dolivero@matem.unam.mx}

\author[C.~O'Neill]{Christopher O'Neill}
\address{Mathematics Department\\San Diego State University\\San Diego, CA 92182}
\email{cdoneill@sdsu.edu}

\author[S.~Zerbib]{Shira Zerbib}
\address{Department of Mathematics\\University of Michigan\\Ann Arbor, MI 48109}
\email{zerbib@umich.edu}

% \subjclass[2010]{Primary: 20M14, 05E40.}

% \keywords{numerical semigroup; computation; quasipolynomial}

%\date{\today}

\begin{abstract}
An \emph{$r$-segment hypergraph} $H$ is a hypergraph whose edges consist of $r$~consecutive integer points on line segments in $\mathbb{R}^2$.  In this paper, we bound the chromatic number $\chi(H)$ and covering number $\tau(H)$ of hypergraphs in this family, uncovering several interesting geometric properties in the process.  We conjecture that for $r \ge 3$, the covering number $\tau(H)$ is at most $(r - 1)\nu(H)$, where $\nu(H)$ denotes the matching number of $H$.  We prove our conjecture in the case where $\nu(H) = 1$, and provide improved (in fact, optimal) bounds on $\tau(H)$ for $r \le 5$.  We also provide sharp bounds on the chromatic number $\chi(H)$ in terms of $r$, and use them to prove two fractional versions of our conjecture.  
\end{abstract}

\maketitle

% \setcounter{tocdepth}{1}
% \tableofcontents

%%%%%%%%%%%%%%%%%%%%%%%%%%%%%%%%%%%%%%%%%%%%%%%%%%%%%%%%%%%%%%%%%%%%%%%%%
\section{Introduction}%%%%%%%%%%%%%%%%%%%%%%%%%%%%%%%%%%%%%%%%%%%%%%%%%%%
\label{sec:intro}%%%%%%%%%%%%%%%%%%%%%%%%%%%%%%%%%%%%%%%%%%%%%%%%%%%%%%%%
%raggedbottom%%%%%%%%%%%%%%%%%%%%%%%%%%%%%%%%%%%%%%%%%%%%%%%%%%%%%%%%%%%%

The combinatorial and geometrical properties of families of sets in Euclidean space has been a thriving area of research in discrete geometry and combinatorics for many decades.  It was initiated by the classical theorem of Helly \cite{helly}, asserting that if $\mathcal F$ is a family of convex sets in $\RR^d$ in which every $d + 1$ members have nonempty intersection, then some point lies in every set in $\mathcal F$.  This landmark result led to the study of \emph{covering numbers} (also sometimes called \emph{piercing numbers} or \emph{hitting numbers}) of families of sets, that is, the minimal number of points needed to ``pierce'' every set in a family given some local intersection property \cite{AK, Eckhoff, HD, 432, KST}.  Helly's theorem is sharp in general, but improved bounds on covering numbers can be obtained if one restricts the convex sets in question.  Indeed, extensive work has been done on the covering numbers of families of disks, boxes, line segments, homothets of centrally symmetric bodies, and other convex sets \cite{CSZ, Danzer, DanzGrum, GL, HMO,  GZ,  kaiser, karasev, KT}, as well as in discrete settings, where the vertex set lies in the integer lattice (of some fixed semi-algebraic group) \cite{AverWais, DLOR,  Doi, halman}. Determining covering numbers in discrete arrangements is often harder than in the continuous case, as the covering sets of points are required to lie in the lattice as well.  

Covering problems have also been studied in the combinatorial setting of \emph{hypergraphs}, that is, a pair $H = (V,E)$ of a \emph{vertex set} $V$ and an \emph{edge set} $E \subset 2^V$.  In this context, one generally wants to obtain an upper bound on the \emph{covering number} $\tau(H)$  in terms of its \emph{matching number} $\nu(H)$ (that is, the maximum number of pairwise disjoint edges in $H$).  If $H$ is an \emph{$r$-uniform} hypergraph (that is, if every edge has exactly $r$ vertices), then one such bound is $\tau(H) \le r\nu(H)$, as the union of the edges in any maximal matching is a cover of $H$ of size $r\nu(H)$ (we will refer to this bound as the \emph{trivial bound}).  

The trivial bound cannot be improved in general, as is exemplified by the disjoint union of copies of the projective plane of uniformity $r$.  For certain families of $r$-uniform hypergraphs, however, the trivial bound can be improved.  K\"onig's theorem \cite{konig}, for instance, asserts that $\tau(H) = \nu(H)$ in the case where $r = 2$ and $H$ is a bipartite graph.  A longstanding conjecture of Ryser (see \cite{aharoni}) generalizes K\"onig's theorem and states that any $r$-partite hypergraph $H$ satisfies $\tau(H) \le (r - 1)\nu(H)$.  Ryser's conjecture was proved by Aharoni \cite{aharoni} in the $r = 3$ case, and remains open for $r \ge 4$; see also~\cite{AZ, haxell, tuza}.  

In this paper, we introduce a new family of $r$-uniform hypergraphs whose edges are comprised of consecutive integer points on a line.  The covering number of such a hypergraph coincides with the notion of covering number in discrete versions of Helly-type theorems, as each edge is the intersection of a convex set (in this case, a line segment) with the integer lattice, and any covering set consists only of lattice points.  
% consisting of $r$ consecutive points on the intersection of a line in $\mathbb{R}^2$ with $\mathbb{Z}^2$ where no two of this segments lie on the same line. More formally:

\begin{definition}\label{d:maindefinition}
Let $r \ge 2$ be an integer. A hypergraph $H = (V,E)$ on vertex set $V \subset \mathbb{Z}^2$ is called an \emph{$r$-segment hypergraph} if
\begin{enumerate}[(i)]
\item
every edge in $E$ consists of $r$ consecutive integer points on some line in $\mathbb{R}^2$ and 

\item
every line in $\mathbb{R}^2$ contains at most one edge of $H$. 

\end{enumerate}
\end{definition}

This new family of hypergraphs has proven to be a fruitful source for challenging combinatorial questions, the answers to which have yielded surprising geometrical properties.  Much of the work in this paper centers around the following conjectured bound on the covering numbers of $r$-segment hypergraphs, one which matches the Ryser conjecture for $r$-partite hypergraphs.  

\begin{conj}\label{conj:mainconjecture}
If $H$ is an $r$-segment hypergraph with $r \ge 3$, then 
$$\tau(H) \le (r - 1)\nu(H).$$
\end{conj}

Our main results in this direction are as follows.  

\begin{enumerate}[(a)]
\item 
We prove Conjecture~\ref{conj:mainconjecture} in the case where $H$ is \emph{intersecting}, i.e.\ when $\nu(H) = 1$ (Theorem~\ref{t:isolatedvertex}).  Our proof is highly geometric, using areas of bounded regions to argue that any such hypergraph must have a vertex that lies in only one edge.  

\item 
We prove two ``fractional'' versions of Conjecture~\ref{conj:mainconjecture} (Theorem~\ref{t:fractional}), a common alternative in the literature when bounding the covering number of a particular hypergraph family proves difficult \cite{furedi, Krivelevich, lovaszpartite}.  More specifically, we show that for any $r$-segment hypergraph $H$, the ratios $\tau(H)/\nu^*(H)$ and $\tau^*(H)/\nu(H)$ do not exceed $r - 1$, where $\nu^*(H)$ and $\tau^*(H)$ are the fractional matching number and fractional covering number, respectively (formal definitions are given in Section~\ref{sec:fractional}).  

\end{enumerate}

The remaining results in this paper are related to those above.  In Section~\ref{sec:coloring}, we provide sharp bounds in terms of $r$ on the \emph{chromatic numbers} of $r$-segment hypergraphs (that is, the minimal number of colors needed to color the vertices so that no edge is monochromatic).  
% Our proof provides a constant-time algorithm for a $k$-coloring of an $r$-segment hypergraph, with $k = 4$, $k = 3$ and $k = 2$ in the cases $r\ge 2$, $r\ge 3$ and $r\ge 4$ respectively.  The key idea is to examine the coordinates of the vertices modulo $k$ for the appropriate $k$. 
% that $\chi(H) \le 4$ when $r = 2$, $\chi(H) \le 3$ when $r = 3$, and $\chi(H) = 2$ when $r \ge 4$
Although initially intended as a means to prove the fractional results in Section~\ref{sec:fractional}, there were enough interesting tangential questions concerning hypergraph colorings to warrant independent investigation, 
% Even this tangential question produced several interesting questions related to planar coloring, 
a testament to the wealth of interesting questions concerning $r$-segment hypergraphs.  
Additionally, we dedicate much of Section~\ref{sec:sharperbounds} to refining the bound on $\tau(H)$ in the case where $H$ is an intersecting $r$-segment hypergraph.  We conjecture a sharper bound in this case, and derive a surprising geometric lemma to verify it in the case $r \le 5$.

%%%%%%%%%%%%%%%%%%%%%%%%%%%%%%%%%%%%%%%%%%%%%%%%%%%%%%%%%%%%%%%%%%%%%%%%%
\section{Coloring $r$-segment hypergraphs}%%%%%%%%%%%%%%%%%%%%%%%%%%%%%%%
\label{sec:coloring}%%%%%%%%%%%%%%%%%%%%%%%%%%%%%%%%%%%%%%%%%%%%%%%%%%%%%
%raggedbottom%%%%%%%%%%%%%%%%%%%%%%%%%%%%%%%%%%%%%%%%%%%%%%%%%%%%%%%%%%%%

A hypergraph $H$ is \emph{$k$-colorable} for a positive integer $k$ if one can assign one of $k$ colors to each vertex in such a way that no edge in $H$ is monochromatic, and the \emph{chromatic number} of $H$, denoted $\chi(H)$, is the least number $k$ such that $H$ is $k$-colorable.  

Coloring problems in geometric settings have a long and vibrant history.  Particularly noteworthy examples include the 4-color theorem for planar graphs and the question of how to color every point in $\RR^2$ so that no two points that are unit distance apart are colored by the same color \cite{degrey}.  
% (equivalently, to find the chromatic number of the graph $G$ with vertex set $\RR^2$ in which any two points that are unit distance apart are connected by an edge)
Chromatic numbers also arise in the study of intersection graphs of families of sets in $\RR^d$, where, for instance, covering numbers of boxes \cite{GL,Eppstein,PT} have been characterized in terms of the chromatic number of complements of odd cycles \cite{BaOl,woe}.  Chromatic numbers of certain families of $r$-uniform hypergraphs are also of interest, such as the family of Kneser hypergraphs $KG^r(n,k)$, whose vertices are the $k$-element subsets of $\{1, \ldots, n\}$ and whose edges are collections of $r$ pairwise disjoint subsets \cite{AFL,kneser, lovaszkneser}.

In this section, we give an optimal upper bound on the chromatic numbers of \mbox{$r$-segment} hypergraphs in terms of $r$ (Theorem~\ref{t:weakcoloring}).  Our proof uses geometric projections to explicitly construct proper colorings (Proposition~\ref{p:coloringprojection}).  We also demonstrate that the bounds in Theorem~\ref{t:weakcoloring} are sharp (Examples~\ref{e:coloringr2sharp} and~\ref{e:coloringr3sharp}).  

Like many other hypergraph families in the literature \cite{GMA, NogaAlon, MPa}, $r$-uniform hypergraphs are \emph{linear}, meaning any two edges intersect in at most one vertex.  It is worth noting that in general, the problem of deciding whether an $r$-uniform hypergraph $H$ is $k$-colorable is NP-hard, even in the case where $H$ is linear \cite{brown, Lo, PR}. 

% \begin{definition}\label{d:coloringprojection}
For each integer $k \ge 2$, let $\ZZ_k$ denote the additive group $\ZZ/k\ZZ = \{\ol 0, \ol 1, \ldots, \ol{k-1}\}$.  In what follows, denote by $Z_k$ the $k$-uniform hypergraph with vertex set $V(Z_k) = \ZZ_k^2$ and edge set
$$E(Z_k) = \{\ol u + \ZZ_k \ol v : u, v \in \ZZ_{k}^2, \, |\ol u + \ZZ_k \ol v| = k\}.$$
% \end{definition}

\begin{prop}\label{p:coloringprojection}
Fix integers $r \ge k \ge 2$ and an $r$-segment hypergraph $H$.  The image of any edge of $H$ under the projection  
$$\begin{array}{rcl}
P_k:\ZZ^2 &\longrightarrow& \ZZ_k^2 \\
(v_1, v_2) &\longmapsto& (\ol v_1, \ol v_2)
\end{array}$$
is an edge in $Z_k$.  In particular, $\chi(H) \le \chi(Z_k)$.  
\end{prop}

\begin{proof}
Fix an edge $e \in E(H)$ with endpoint $u \in \ZZ^2$ and vector $v = (v_1, v_2) \in \ZZ^2$ pointing from $u$ to the next integer point in $e$.  Since $e = u + \{0,1,\dots, r-1\}v$, the image of $e$ under $P_k$ is $\ol u + \ZZ_k \ol v$, so to complete the proof, we claim the images of any $k$ sequential vertices in $e$ are distinct vertices of $Z_k$, or, equivalently, that no nonzero element of $\ZZ_k$ annihilates $\ol v$.  The fact that $u$ and $u + v$ are two consecutive integer points on a line in $\R^2$ implies $\gcd(v_1,v_2) = 1$, so the claim follows from the fact that $\ann(\ol v) = \ann(\ol v_1) \cap \ann(\ol v_2)$ has nonzero elements only if $\gcd(v_1, v_2) > 1$. 
\end{proof}

We are now ready to prove our main theorem on colorations. 

\begin{thm}\label{t:weakcoloring}
Let $H$ be an $r$-segment hypergraph.
\begin{enumerate}[(a)]
\item \label{t:weakcoloring:r2}
If $r=2$ then $\chi(H) \le 4$.

\item \label{t:weakcoloring:r3}
If $r=3$ then $\chi(H) \le 3$.

\item \label{t:weakcoloring:r4}
If $r \ge 4$ then $\chi(H) = 2$.
\end{enumerate}
Additionally, the inequalities in parts~\eqref{t:weakcoloring:r2} and~\eqref{t:weakcoloring:r3} are sharp.  
\end{thm}

\begin{proof}[Proof of Theorem \ref{t:weakcoloring}]
By Proposition~\ref{p:coloringprojection}, it suffices to show $\chi(Z_2) \le 4$, $\chi(Z_3) \le 3$ and $\chi(Z_4) \le 2$.  The first inequality is clear since $Z_2$ is a complete graph on 4 vertices, and the other two follow from the colorings given in Figure~\ref{f:weak3coloring} and~\ref{f:weak2coloring}, respectively (one can manually verify that no monochromatic edges are present). The remaining claims about the sharpness of parts~\eqref{t:weakcoloring:r2} and~\eqref{t:weakcoloring:r3} of the theorem follow from Examples~\ref{e:coloringr2sharp} and~\ref{e:coloringr3sharp} below, respectively.  
\end{proof}

\begin{figure}[t!]
\begin{center}
\begin{subfigure}[t]{0.45\textwidth}
\begin{center}
\includegraphics[width=1.5in]{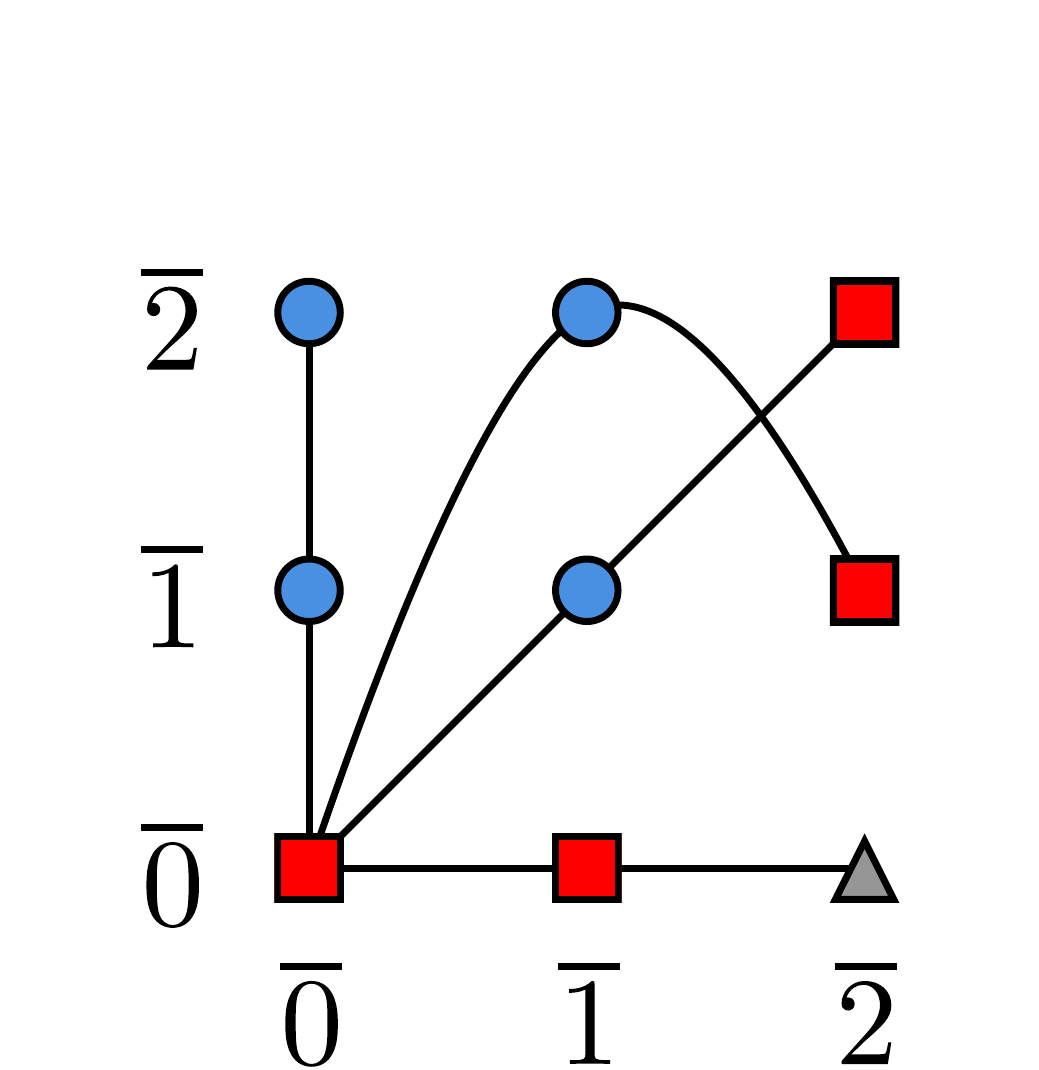}
\end{center}
\caption{}
\label{f:weak3coloring}
\end{subfigure}
\hspace{0.02\textwidth}
\begin{subfigure}[t]{0.45\textwidth}
\begin{center}
\includegraphics[width=1.5in]{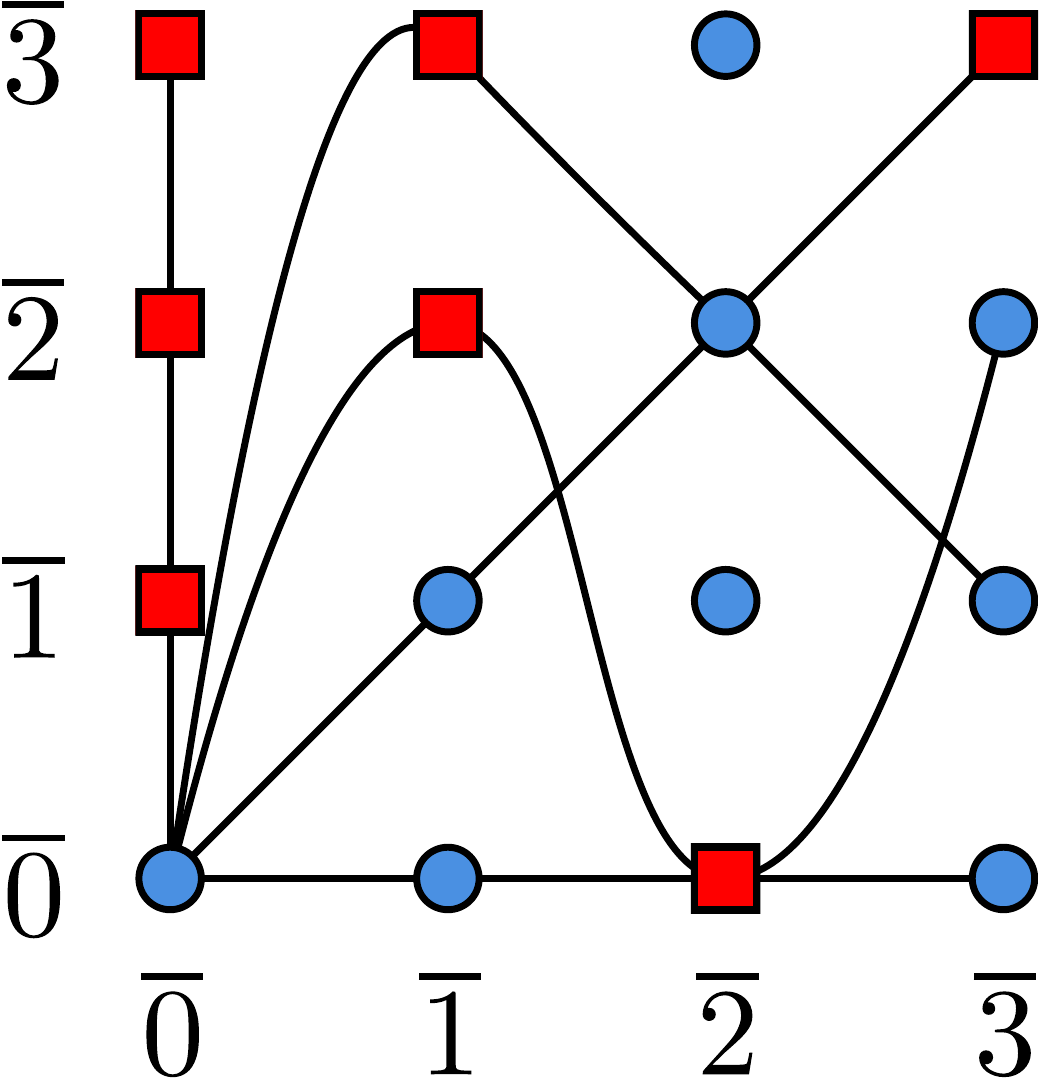}
\end{center}
\caption{}
\label{f:weak2coloring}
\end{subfigure}
\end{center}
\caption{Colorings of $Z_3$ (left) and $Z_4$ (right) for the proof of Theorem~\ref{t:weakcoloring}.  In each, every edge through the point $(\ol 0, \ol 0)$ is also depicted.}
\end{figure}

\begin{example}\label{e:coloringr2sharp}
The $2$-segment hypergraph $H$ with vertex set 
$$V(H) = \{(0,0), (1,1), (1,2), (2,1)\}$$
depicted in Figure~\ref{f:coloringk4} is the complete graph on 4 vertices, and thus has $\chi(H) = 4$.  
\end{example}

\begin{example}\label{e:coloringr3sharp}
The following $3$-segment hypergraph $H$ has $\chi(H) = 3$.  We~begin by constructing the hypergraph $C$ depicted in Figure~\ref{f:coloringcube} with vertex set $\{0,1,2\}^3$.  Start with the 1-skeleton of the cube with vertex set $\{0,2\}^3$.  Next, add edges so that each side of the cube forms a square with 3 horizontal edges, 3 vertical edges, and 2 diagonal edges (see Figure~\ref{f:coloringcubeface}).  Finally, include edges so that the center layer (i.e.\ $z = 1$) matches Figure~\ref{f:coloringcubeface} as well.  This yields 40 edges:\ 8~within each layer parallel to the $xy$-plane, 8~``vertical'' edges parallel to the $z$-axis, and 2~diagonal edges within each vertical side.  One can check via computation that $C$ is not 2-colorable (a~brute-force implementation in Python takes approximately 1 minute).  
% There is probably a slick combinatorial argument as well, obtainable by someone more clever than me.  

Next, let $H$ denote the image of $C$ under the projection map 
$$\begin{array}{rcl}
\ZZ^3 &\longrightarrow& \ZZ^2 \\
(x, y, z) &\longmapsto& (x + 17z, y + 29z).
\end{array}$$
The lower left corners of the squares comprising the bottom ($z = 0$), middle ($z = 1$), and top $(z = 2)$ layers of $C$ are $(0,0)$, $(17,19)$, and $(34,38)$, respectively.  It suffices to check that the image of each ``vertical'' edge in $H'$ still only contains 3 integer points, but the vector $v$ in the direction of any such edge from the bottom layer ($z = 0$) to the middle layer ($z = 1$) either has $x$-coordinate 17 or $y$-coordinate 29, ensuring $\gcd(v_1, v_2) = 1$.  As such, $H$ is a 3-segment hypergraph with chromatic number $3$.  
\end{example}

\begin{figure}[t!]
\begin{center}
\begin{subfigure}[t]{0.23\textwidth}
\begin{center}
\includegraphics[width=\textwidth]{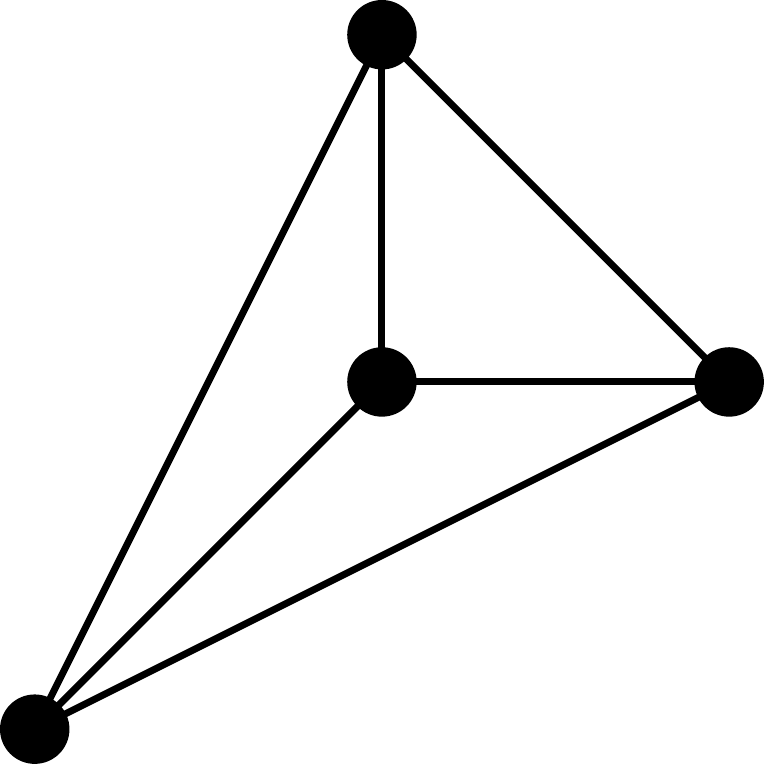}
\end{center}
\caption{}
\label{f:coloringk4}
\end{subfigure}
\hspace{0.04\textwidth}
\begin{subfigure}[t]{0.23\textwidth}
\begin{center}
\includegraphics[width=\textwidth]{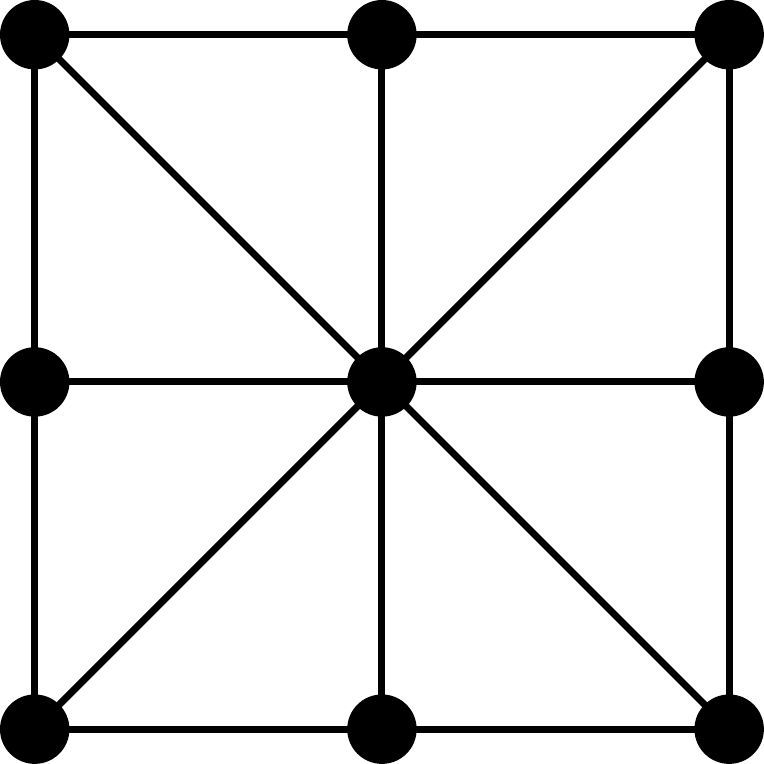}
\end{center}
\caption{}
\label{f:coloringcubeface}
\end{subfigure}
\hspace{0.04\textwidth}
\begin{subfigure}[t]{0.4\textwidth}
\begin{center}
\includegraphics[width=\textwidth]{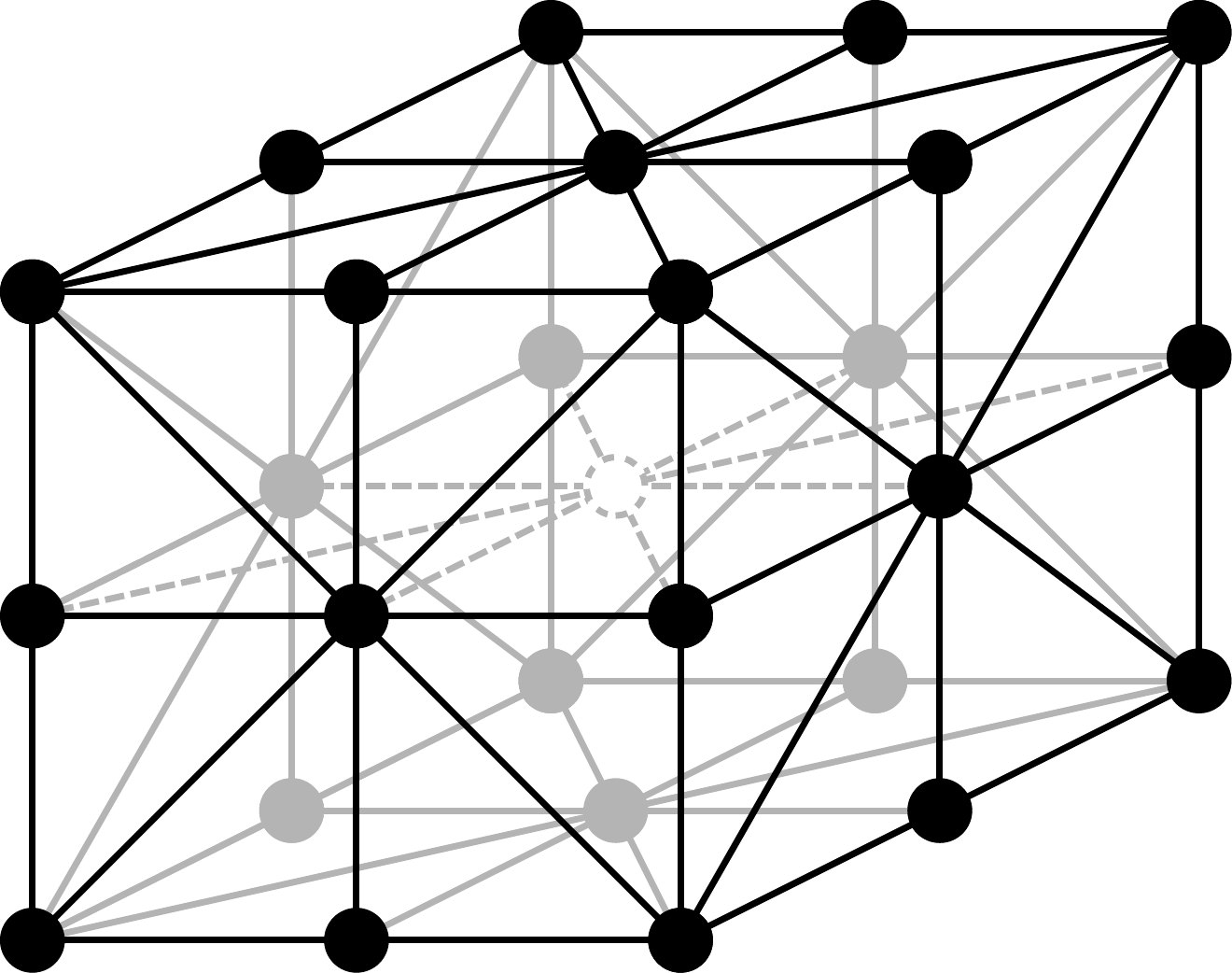}
\end{center}
\caption{}
\label{f:coloringcube}
\end{subfigure}
\end{center}
\caption{Hypergraphs from Examples~\ref{e:coloringr2sharp} and~\ref{e:coloringr3sharp}.}
\end{figure}

One immediate consequence of Theorem~\ref{t:weakcoloring} is a bound on $\tau(H)$ in terms of $V(H)$ for any $r$-segment hypergraph $H$.  

\begin{cor}\label{c:weakcoloring}
Let $H = (V,E)$ be an $r$-segment hypergraph. 
\begin{enumerate}[(a)]
\item If $r = 2$, then $\tau(H) \le \frac{3}{4}|V|$. 
\item If $r = 3$, then $\tau(H) \le \frac{2}{3}|V|$.
\item For every $r \ge 4$, $\tau(H) \le \frac{1}{2}|V|$.
\end{enumerate}  
\end{cor}

\begin{proof}
Suppose $H$ is $k$-colorable.  For any color $c$, the set of all vertices not colored by $c$ is a cover. By the pigeon hole principle, some color must label at most $1/k$ of the vertices of $H$, so $\tau(H) \le \frac{k-1}{k}|V|$.  The result now follows from Theorem~\ref{t:weakcoloring}.  
\end{proof}

\begin{remark}\label{r:r2realization}
One surprising consequence of Theorem~\ref{t:weakcoloring} is that some graphs are not realizable as 2-segment hypergraphs.  For example, $K_5$ (the complete graph on 5~vertices) is the smallest complete graph that cannot be realized as a 2-segment hypergraph, and is also the smallest non-planar complete graph.  On the other hand, $K_{3,3}$ (the complete bipartite graph with 3 vertices of each color) is realizable as a 2-segment hypergraph (in fact, any bipartite graph is realizable).  
\end{remark}

Schnyder proved that every planar graph can be embedded so that every vertex is an integer point and every edge is a straight line segment \cite{Walter}.  This, together with Remark~\ref{r:r2realization}, prompts the following.  

\begin{prob}\label{prob:r2realization}
Which graphs are realizable as 2-segment hypergraphs?  
\end{prob}

%%%%%%%%%%%%%%%%%%%%%%%%%%%%%%%%%%%%%%%%%%%%%%%%%%%%%%%%%%%%%%%%%%%%%%%%%
\section{Covering intersecting $r$-segment hypergraphs}%%%%%%%%%%%%%%%%%%
\label{sec:isolatedvertex}%%%%%%%%%%%%%%%%%%%%%%%%%%%%%%%%%%%%%%%%%%%%%%%
%raggedbottom%%%%%%%%%%%%%%%%%%%%%%%%%%%%%%%%%%%%%%%%%%%%%%%%%%%%%%%%%%%%

The main result of this section is Theorem~\ref{t:isolatedvertex}, in which we prove Conjecture~\ref{conj:mainconjecture} for intersecting $r$-segment hypergraphs.  First, we give two examples demonstrating that no hypotheses in Conjecture~\ref{conj:mainconjecture}  can be dropped.  

\begin{example}\label{e:conjecturer2}
Conjecture~\ref{conj:mainconjecture} does not hold for $r = 2$, as exemplified by the complete $2$-segment hypergraph $R$ on the vertex set $V(R) = \{(0,0),(1,0),(0,1)\}$, which has $\nu = 1$ and $\tau = 2$.
\end{example}

\begin{example}\label{e:nonsequential}
Without the ``consecutive'' assumption in Definition~\ref{d:maindefinition}, Conjecture~\ref{conj:mainconjecture} would fail to hold.  Consider, for instance, the hypergraph $S$ with vertex set
$$V(S) = \{(-2,0),(0,0),(2,0),(0,2),(-1,3),(1,3),(0,6)\}$$
depicted in Figure~\ref{f:nonsequential}.  
This configuration resembles the classical non-Fano plane and has $\nu(H) = 1$ but $\tau(H) = 3$.  
\end{example}

In what follows, an \emph{isolated point} in a hypergraph is a vertex lying in only one	edge.  

\begin{thm}\label{t:isolatedvertex}
If $H$ is an intersecting $r$-segment hypergraph with $r \ge 3$, then $H$ contains an isolated vertex. In particular, $\tau(H) \le r - 1$.    
\end{thm}

\begin{proof}
Assume for contradiction that $H$ contains no isolated vertices.   If all of the edges in $H$ intersect at a vertex, then the theorem is trivial.  Otherwise, there are three edges $e_1, e_2, e_3 \in E(H)$ in $H$ whose corresponding line segments form a triangle $T = \{a,b,c\}$ with $a = e_1 \cap e_2$, $b = e_2 \cap e_3$ and $c = e_3 \cap e_1$.  Suppose further that $T$ has minimal area among all triangles formed by edges in $H$.  The minimality of the area of $T$ implies that $T$ contains no other integer points in its boundary, and since $r \geq 3$, each edge $e_i$ bounding $T$ contains at least one more vertex.  

First, suppose that each of the vertices $a$, $b$, and $c$ is an endpoint of one of the edges $e_1$, $e_2$, and $e_3$, and let $p_i$ denote the next sequential vertex along each edge $e_i$, as depicted in Figure~\ref{f:oneeach}.  Since $H$ contains no isolated vertices, some other edge $e_{p_1}$ must contain $p_1$, and since any two edges intersect, $e_{p_1}$ must intersect both $e_2$ and $e_3$ at $b$. Similarly, some other edge $e_{p_3}$ contains both $p_3$ and $a$, and must also intersect $e_{p_1}$ at some integer point $q$ between $p_1$ and $b$.  Now, observe that $T$ and the triangle $\{a,p_1,b\}$ have the same area, meaning that the triangle $\{a,p_1,q\}$ has smaller area than $T$, which is a contradiction.  

In all the remaining cases, one of the vertices $a$, $b$, and $c$ is not an endpoint of any of the edges $e_1$, $e_2$, and $e_3$. Without loss of generality, assume this vertex is $a$.  Let $p_1$ and $p_2$ denote the vertices adjacent to $a$ on $e_1$ and $e_2$, respectively, as in Figure~\ref{f:parallel}. As before, since $H$ has no isolated vertices, some other edges $e_{p_1}$ and $e_{p_2}$ must contain $p_1$ and $p_2$, respectively.  Since $H$ is intersecting and the line $\ol{p_1p_2}$ is parallel to $e_3$, the edges $e_{p_1}$ and $e_{p_2}$ must be distinct.  Moreover, $e_{p_1}$ and $e_{p_2}$ cannot intersect $b$ and $c$, respectively, as this would make $e_{p_1}$ and $e_{p_2}$ parallel (see Figure~\ref{f:parallel}).  Thus we may assume that $e_{p_2}$ intersects $e_3$ at a point past vertex $b$, as depicted in Figure~\ref{f:minimum}.  

If $e_{p_2}$ intersects $e_3$ at the vertex  adjacent to $b$, then $e_{p_2}$ is parallel to $e_1$.  Therefore, $e_{p_2}$ must intersect $e_1$ at a vertex closer to $p_1$ than to $a$.  Labeling the vertices on $e_1$ as $q_0 = p_1,  q_1, q_2, \ldots$ according to their distance from $a$, assume that $e_{p_2}$ is the edge containing $p_2$ that contains $q_j$ for $j \ge 1$ minimal.  Since $q_{j-1}$ is not isolated, it must be contained in some additional edge $e_{q_{j-1}}$ that, by the minimality of $j$, does not contain $p_2$.  However, the triangle $\{p_2,q_{j-1},q_j\}$ has the same area as $T$, and $e_{q_{j-1}}$ must divide this triangle in order to intersect $e_{p_2}$, contradicting the minimality of $T$.  
\end{proof}

\begin{figure}[t!]
\begin{center}
\begin{subfigure}[t]{0.3\textwidth}
\includegraphics[width=2in]{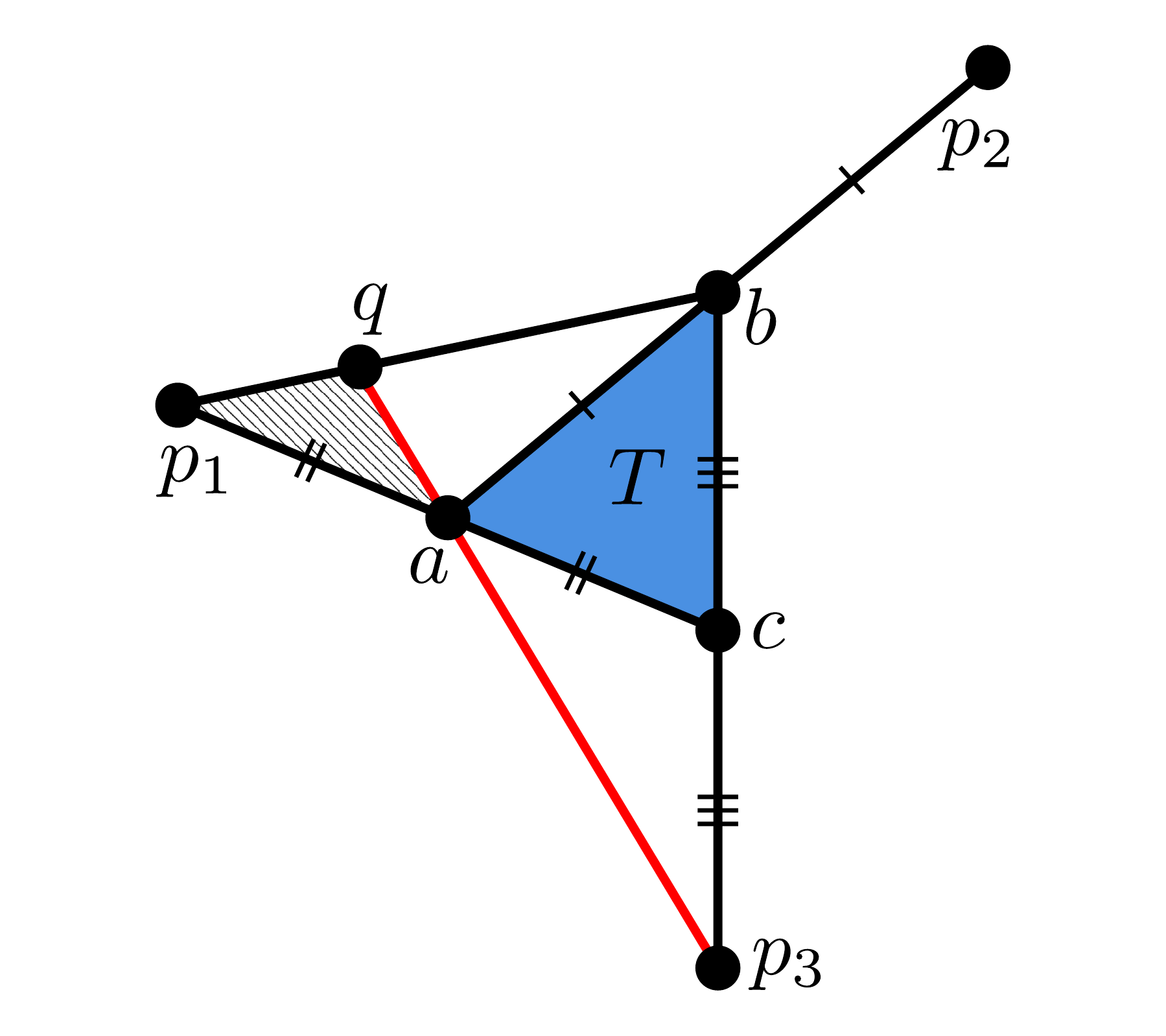}
\caption{}
\label{f:oneeach}
\end{subfigure}
% \hspace{0.05\textwidth}
\begin{subfigure}[t]{0.3\textwidth}
\includegraphics[width=2in]{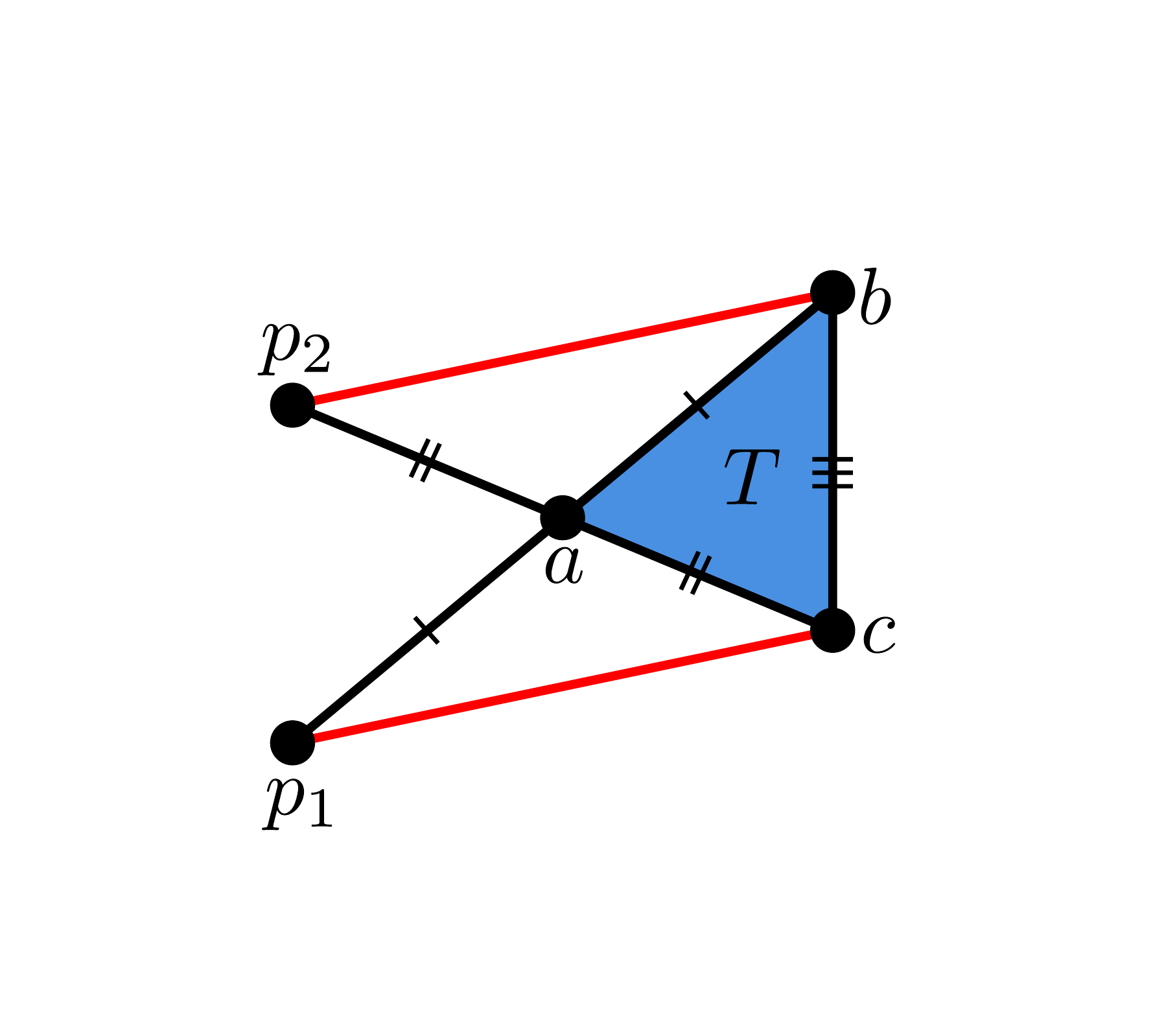}
\caption{}
\label{f:parallel}
\end{subfigure}
\begin{subfigure}[t]{0.3\textwidth}
\includegraphics[width=2in]{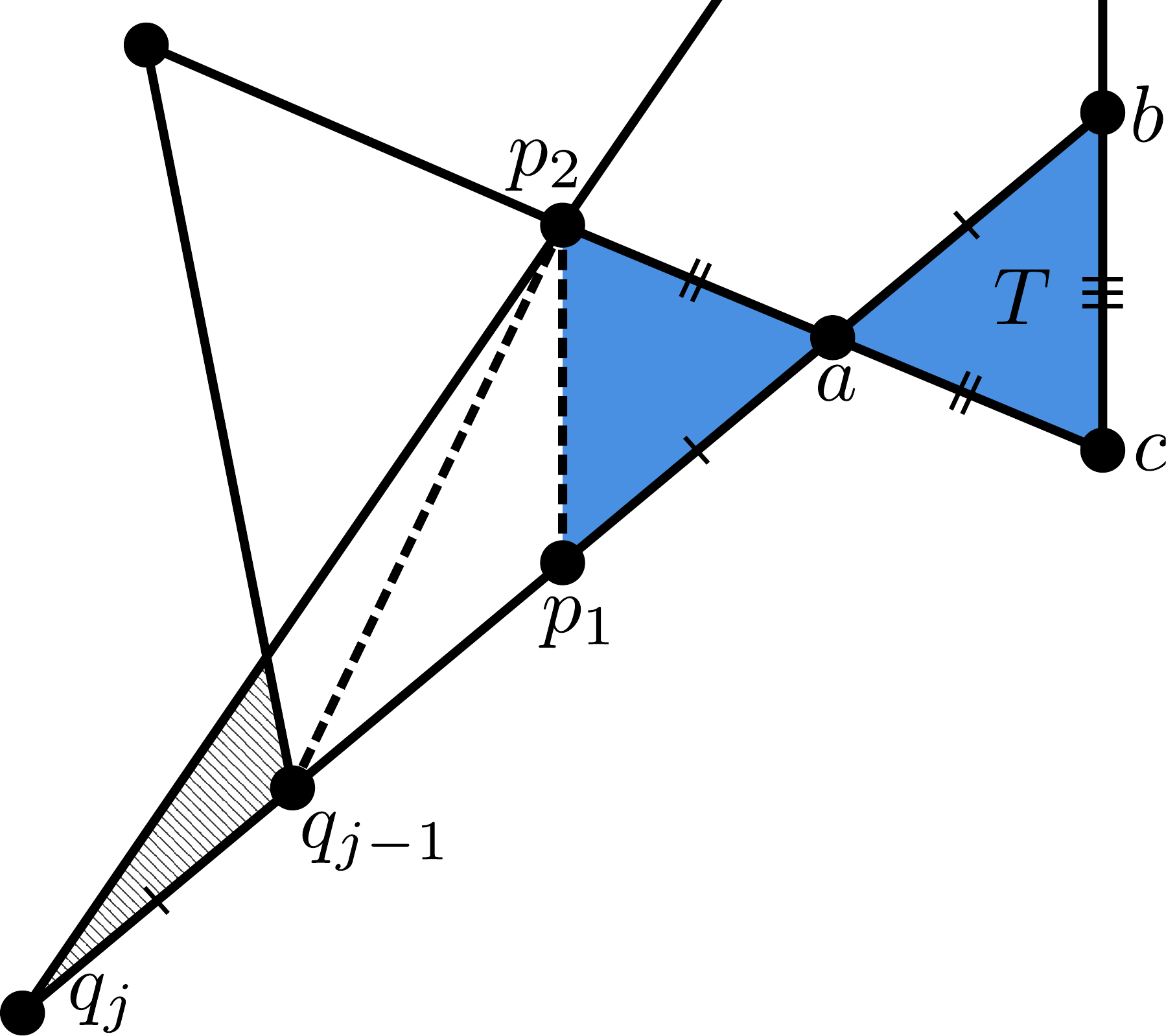}
\caption{}
\label{f:minimum}
\end{subfigure}
\end{center}
\caption{Configurations for the proof of Theorem~\ref{t:isolatedvertex}.}
\end{figure}

\begin{example}\label{e:isolatedvertex}
Theorem~\ref{t:isolatedvertex} need not hold for non-intersecting $r$-segment hypergraphs in general. For example, the 3-hypergraph in Figure~\ref{f:coloringcubeface} has no isolated vertices.  
\end{example}

%%%%%%%%%%%%%%%%%%%%%%%%%%%%%%%%%%%%%%%%%%%%%%%%%%%%%%%%%%%%%%%%%%%%%%%%%
\section{A refined bound for intersecting $r$-segment hypergraphs}%%%%%%%
\label{sec:sharperbounds}%%%%%%%%%%%%%%%%%%%%%%%%%%%%%%%%%%%%%%%%%%%%%%%%
%raggedbottom%%%%%%%%%%%%%%%%%%%%%%%%%%%%%%%%%%%%%%%%%%%%%%%%%%%%%%%%%%%%

Our main goal in this section is to investigate the following refinement of Conjecture~\ref{conj:mainconjecture} for intersecting $r$-segment hypergraphs, which closely resembles the bound in the generalized Tuza conjecture for certain linear hypergraphs (Remark~\ref{r:intersectingboundtuza}).  

\begin{conj}\label{conj:intersectingbound}
If $H$ is an intersecting $r$-segment hypergraph and $r \ge 5$, then 
$$\tau(H) \le \left\lceil \, \frac{r}{2} \, \right\rceil.$$
\end{conj}

We prove two results in support of Conjecture~\ref{conj:intersectingbound}:\ (i) we show that the bound in Conjecture~\ref{conj:intersectingbound}, if true, is sharp (Theorem~\ref{t:intersectinglowerbound}), and (ii) we prove that Conjecture~\ref{conj:intersectingbound} holds in the case $r = 5$ (Theorem~\ref{t:r5intersecting}), confirming a strict departure from the bound in Conjecture~\ref{conj:mainconjecture}.  Note that Conjecture~\ref{conj:intersectingbound} fails to hold for $r < 5$, as Conjecture~\ref{conj:mainconjecture} is tight in this case.  Indeed, adding an isolated vertex to every edge in the 2-segment hypergraph $R$ in Example~\ref{e:conjecturer2} yields a $3$-segment hypergraph with $\tau = 2$, and Example~\ref{e:r4intersecting} demonstrates sharpness for $r = 4$.

\begin{remark}\label{r:intersectingboundtuza}
The upper bound in Conjecture~\ref{conj:intersectingbound} almost matches the bound conjectured by Aharoni and the third author \cite{AZ} for a family of hypergraphs $H^{(r-1)}$ (see \cite{AZ} for formal definitions) generalizing a famous conjecture of Tuza \cite{tuza}.  We believe this is no coincidence, since both hypergraphs share a similar ``strong-linearity" property:\ given any two vertices $u, v$ in $H^{(r-1)}$, there is at most one possible $r$-tuple of elements of $V(H^{(r-1)})$ containing both $u$ and $v$ that can serve as an edge; similarly, in an $r$-segment hypergraph $H$, any edge $e$ is fully determined by the pair of vertices $u, v$ at the two ends of the segment corresponding to $e$, in that there is at most one possible $r$-tuple in $\mathbb{Z}^2$ containing both $u$ and $v$ that can serve as an edge in any $r$-segment hypergraph.  
\end{remark}

\begin{figure}[t!]
\begin{center}
\begin{subfigure}[t]{0.3\textwidth}
\begin{center}
\includegraphics[height=1.5in]{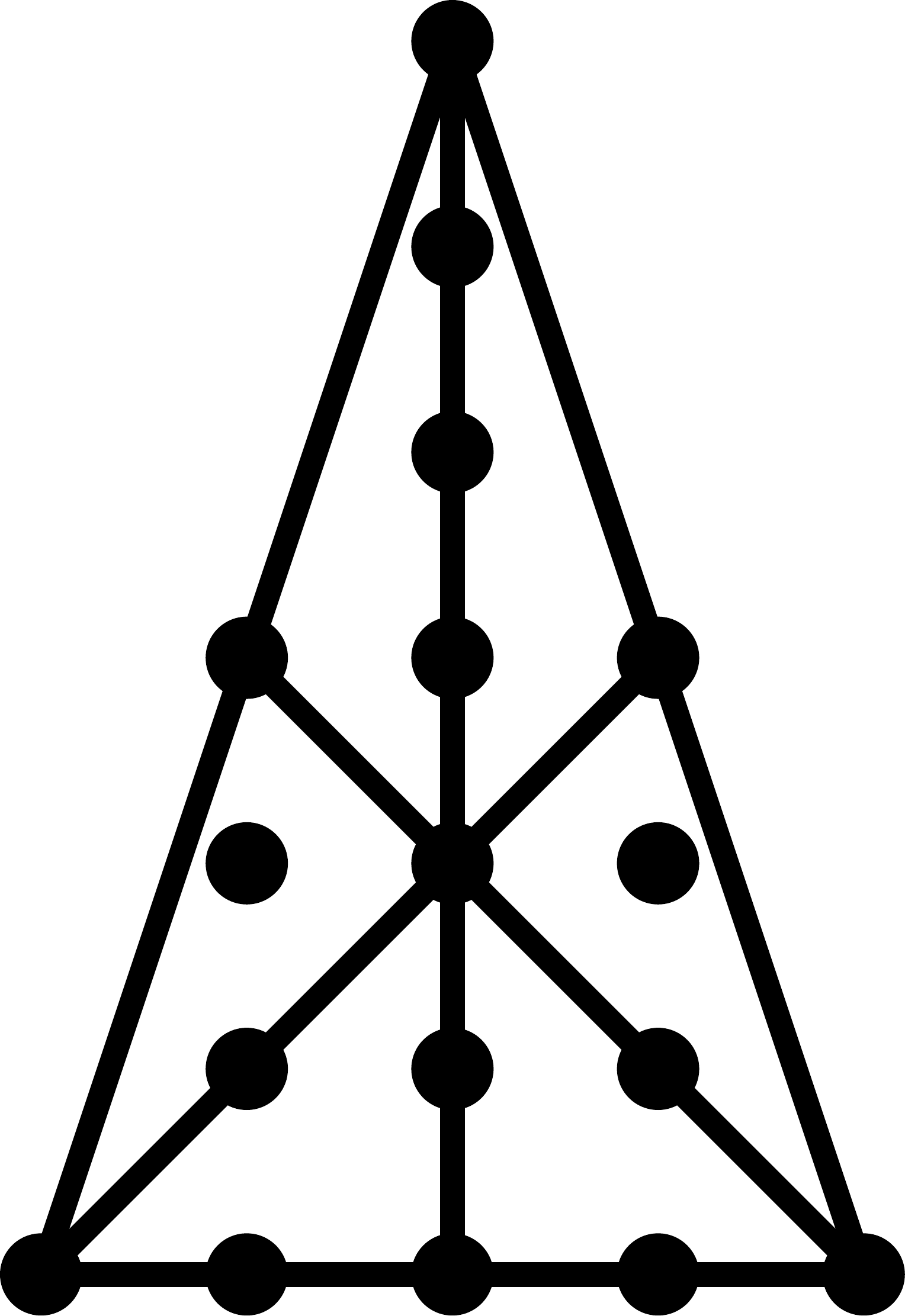}
\end{center}
\caption{}
\label{f:r4intersecting}
\end{subfigure}
\begin{subfigure}[t]{0.3\textwidth}
\begin{center}
\includegraphics[height=1.5in]{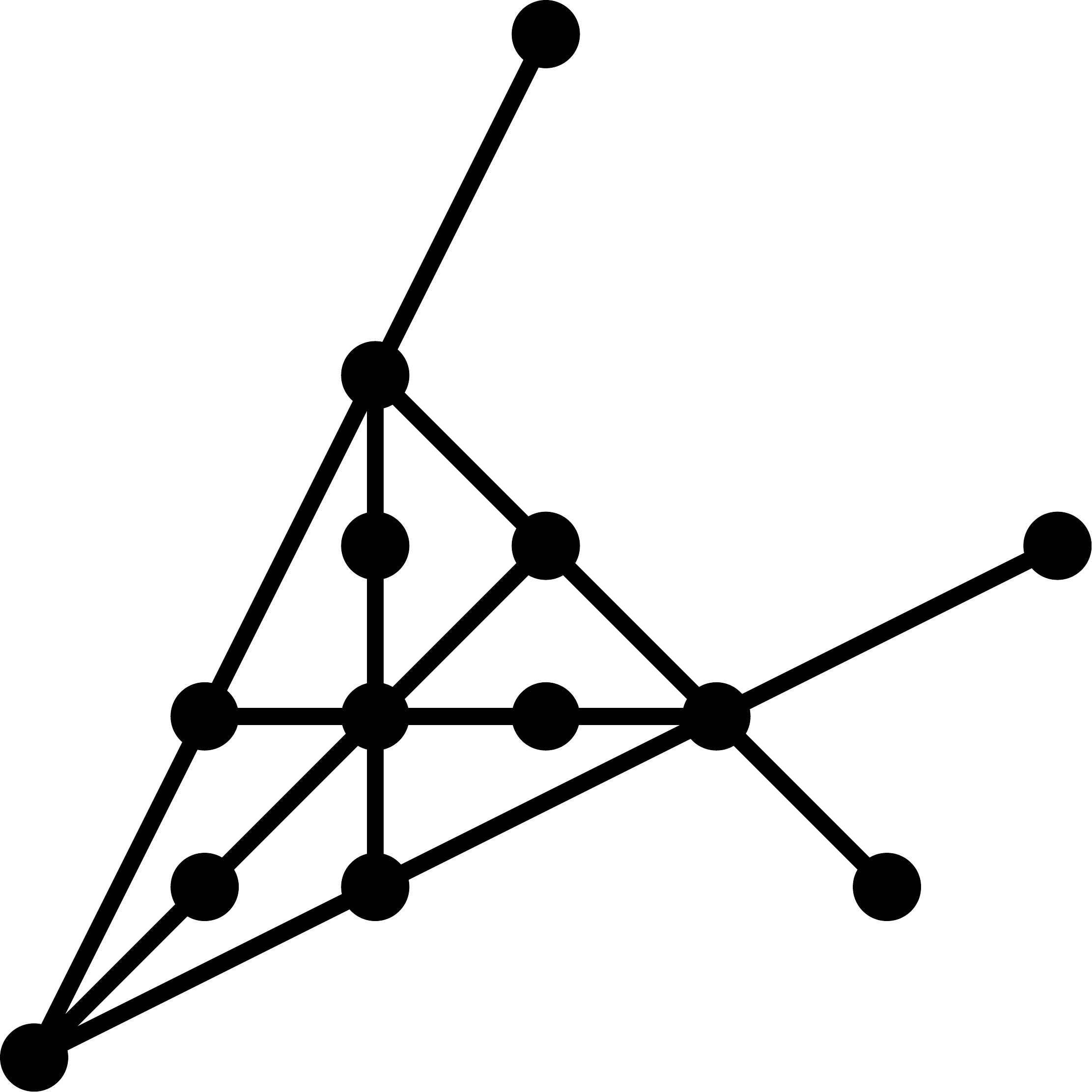}
\end{center}
\caption{}
\label{f:nonsequential}
\end{subfigure}
\begin{subfigure}[t]{0.3\textwidth}
\begin{center}
\includegraphics[width=1.5in]{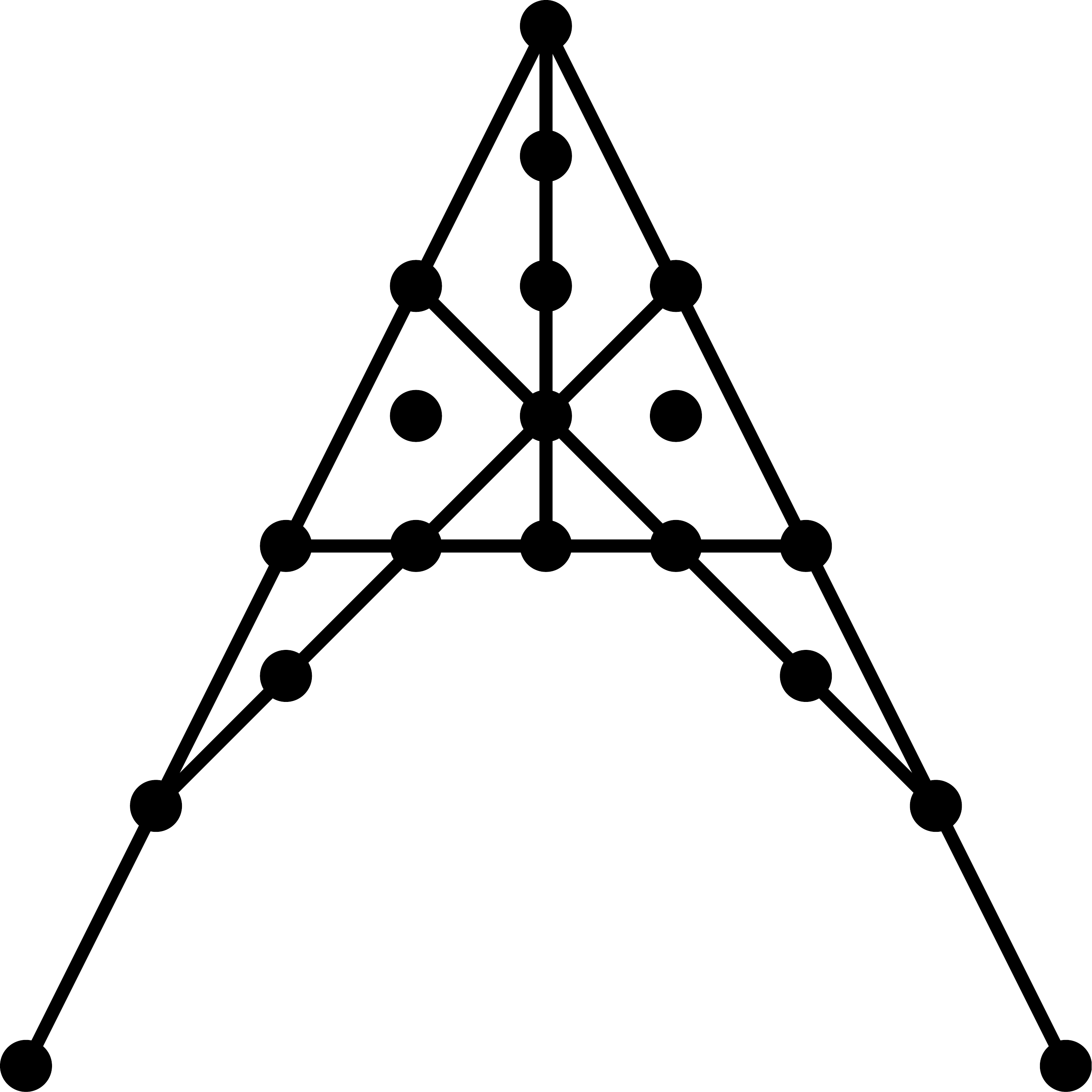}
\end{center}
\caption{}
\label{f:r5intersectingmonsters}
\end{subfigure}
\end{center}
\caption{The configuration lacking the ``sequential'' property from Example~\ref{e:nonsequential} (left), the intersecting $4$-segment hypergraph with $\tau = 3$ from Example~\ref{e:r4intersecting} (middle), and an intersecting 5-segment hypergraph with 6 edges (right).}
\end{figure}

\begin{example}\label{e:r4intersecting}
The hypergraph $H$ depicted in Figure~\ref{f:r4intersecting} is an intersecting $4$-segment hypergraph with covering number $3$. To see this, note that there are $6$ edges in $H$, and any two vertices of degree $3$ share an edge. As such, any two vertices in $H$ cover at most $5$ edges. 

Note that this hypergraph is the $3$-uniform projective plane minus one edge, with an added isolated vertex in every edge.  This construction appears to be special to $r = 4$, in part because of the increased difficulty of obtaining straight-line embeddings of the $r$-uniform projective plane for large $r$.  The authors were unable to obtain a proof, however, as the proof method used in Theorem~\ref{t:r5intersecting} includes too many cases to reasonably check by hand when $r > 5$.  
\end{example}

\begin{thm}\label{t:intersectinglowerbound}
For every integer  $r \ge 5$, there exists an intersecting $r$-segment hypergraph with covering number $\lceil\frac{r}{2}\rceil$.  
\end{thm}

\begin{proof}
Consider the vectors
$$\vec v = (1, r - 1) \text{ and } \vec u = (-1, r - 1).$$
Construct an $r$-segment hypergraph $H$ by starting with a segment $e$ from $(0, 0)$ to $(r - 1) \vec v$ and a segment $f$ from $(0, 0)$ to $(r - 1) \vec u$, and for each $i = 1, \ldots, r - 2$, include a segment $e_i$ from $i \vec v$ to $(r - 1 - i) \vec u$ (connecting $e$ to $f$).  We must show the following.  
\begin{enumerate}[(i)]
\item 
Each segment has exactly $r$ integer points.  Indeed, $e$ and $f$ each have exactly $r$ integer points (one at each integral $x$-coordinate), and each edge $e_i$ has slope 
$$m_i = \frac{i(r - 1) - (r - 1 - i)(r - 1)}{i - (r - 1 - i)(-1)} = 2i - r + 1 \in \ZZ,$$
meaning each of the $r$ points on $e_i$ with integral $x$-coordinate is an integer point.  

\item 
Any two segments intersect at an integer point.  Indeed, segments $e$ and $f$ intersect at the origin, and the endpoints of each edge $e_i$ intersect $e$ and $f$, respectively.  Moreover, two distinct edges $e_i$ and $e_j$ intersect at 
$$x = i + j + 1 - r,$$
since the corresponding $y$-coordinate is
$$\begin{array}{rcl}
m_i(x - i) + i(r - 1)
&=& (2i - r + 1)(i + j + 1 - r - i) + i(r - 1) \\
&=& 2ij - (r - 1)(i + j) + (r-1)^2 \\
&=& (2j - r + 1)(i + j + 1 - r - j) + j(r - 1) \\
&=& m_j(x - j) + j(r - 1).
\end{array}$$

\item 
No 3 segments intersect at the same integer point.  This holds since, on a given edge, each intersection point identified in part~(ii) has a distinct $x$-coordinate.  
\end{enumerate}
From these facts, we conclude $\tau(H) = \big\lceil \frac{1}{2}|E(H)| \big\rceil = \big\lceil \frac{r}{2} \big\rceil$.  
\end{proof}

For the proof of Theorem~\ref{t:r5intersecting}, we first prove Lemma~\ref{l:bratios}. This geometric result were also used by the authors to systematically verify Conjecture~\ref{conj:intersectingbound} for fixed values of $r \ge 6$ in only finitely many cases.  

\begin{figure}
\begin{center}
\includegraphics[width=5in]{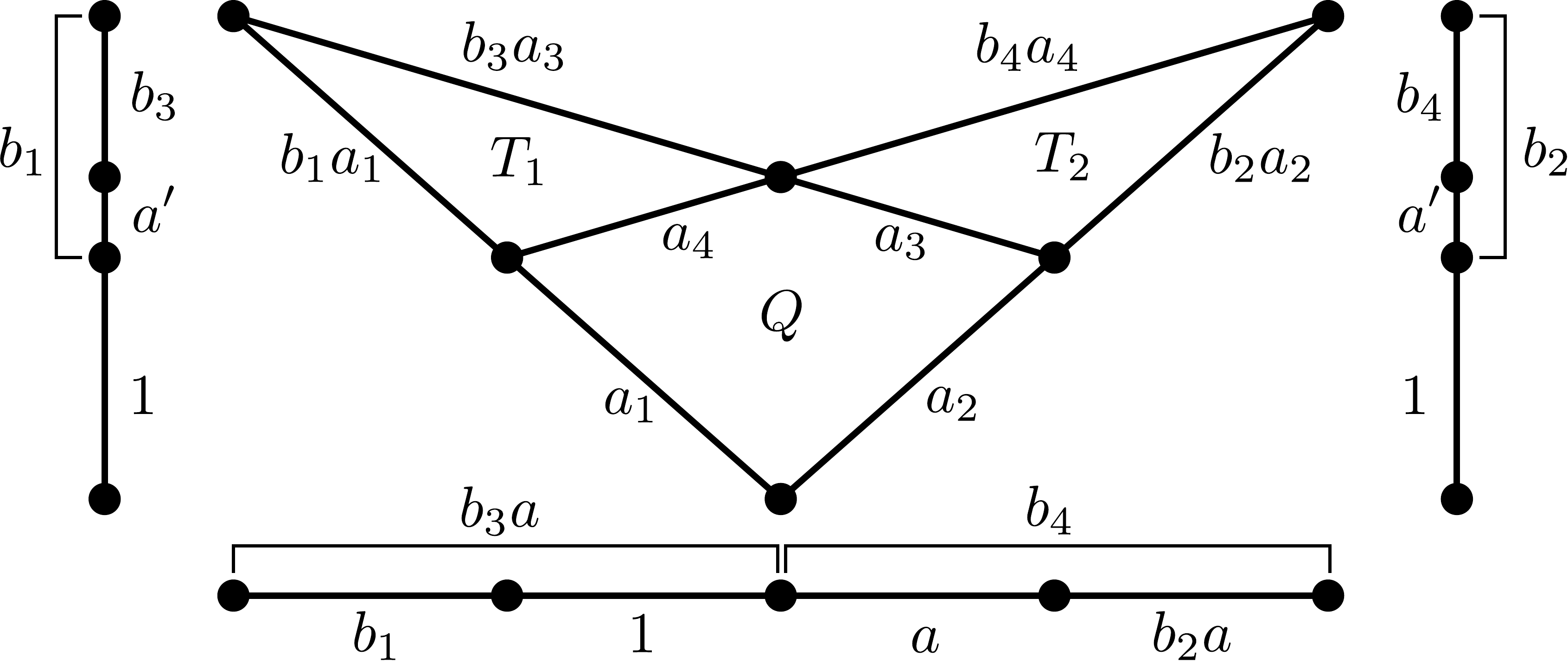}
\end{center}
\caption{The configuration from Lemma~\ref{l:bratios}, together with each projection used in the proof.}
\label{f:bratios}
\end{figure}

\begin{lemma}\label{l:bratios}
Fix a configuration of four pairwise intersecting line segments in $\RR^2$ in which no three intersect at the same point. These segments necessarily bound a quadrilateral $Q$ and two triangles $T_1$ and $T_2$, as in Figure~\ref{f:bratios}.  Denote the lengths of the sides of the quadrilateral $Q$ by $a_1,\ldots, a_4$, and the lengths of the remaining sides of $T_1$ and $T_2$ by $b_1a_1, \ldots, b_4a_4$, as shown in Figure~\ref{f:bratios} (here, we assume only that each labeled length is positive).  The ratios $b_1, \ldots, b_4$ must satisfy the relations
$$b_2b_3 = b_1(b_2 + 1) \qquad \text{and} \qquad b_1b_4 = (b_1 + 1)b_2.$$
In particular, specifying rational values for any two of the $b_i$'s uniquely determines rational values for the other two.  
\end{lemma}

\begin{proof}
Let $A_1, \dots, A_4$, $B_1,\dots, B_4$ denote the given segments in the lemma of lengths $a_1, \dots, a_4$, $a_1b_1,\dots, a_4b_4$, respectively.  Note the ratios $b_1, \ldots, b_4$ are preserved under  rotation and scaling of the entire configuration, as well as under projection onto a line.  

First, consider the projection $P$ along the diagonal of $Q$ connecting the points $A_1 \cap A_2$ and $A_3 \cap A_4$ (the ``vertical'' diagonal in Figure~\ref{f:bratios}).  Scale the original configuration so that $P(A_1)$ (which coincides with $P(A_4)$) has unit length, and let $a$ be the length of $P(A_2) = P(A_3)$.  Comparing projected lengths yields
$$b_3a = b_1 + 1 \qquad \text{and} \qquad (b_2 + 1)a = b_4$$
which can be solved to obtain 
\begin{equation}\label{eq1:ratios}
(b_1 + 1)(b_2 + 1) = b_3b_4.  
\end{equation}
\indent
Next, consider the projection $P'$ along the diagonal of $Q$ connecting the points $A_1 \cap A_4$ and  $A_2 \cap A_3$ (the ``horizontal'' diagonal in Figure~\ref{f:bratios}). Scaling the original configuration so that $P'(A_1) = P'(A_2)$ has unit length and writing $a'$ for the length of $P'(A_3) = P'(A_4)$, we obtain 
$$(b_3 + 1)a'  = b_1 \qquad \text{and} \qquad (b_4 + 1)a' = b_2,$$
yielding 
\begin{equation}\label{eq2:ratios}
b_1(b_4 + 1) = b_2(b_3 + 1).  
\end{equation}

Solving (\ref{eq1:ratios}) for $b_4$ and substituting into (\ref{eq2:ratios}) yields 
$$b_2b_3^2 + (b_2 - b_1)b_3 + b_1(b_1 + 1)(b_2 + 1) = 0,$$
implying
$$b_3 = \frac{(b_1 - b_2) \pm (b_1 + b_2 + 2b_1b_2)}{2b_2}.$$

Disregarding the negative solution yields $b_3 = (b_2 + 1)b_1/b_2$.  Substituting back into (\ref{eq1:ratios}) and (\ref{eq2:ratios}) yields the desired equations. 
\end{proof}

\begin{figure}[t!]
\begin{center}
\begin{subfigure}[t]{0.3\textwidth}
\includegraphics[width=1.8in]{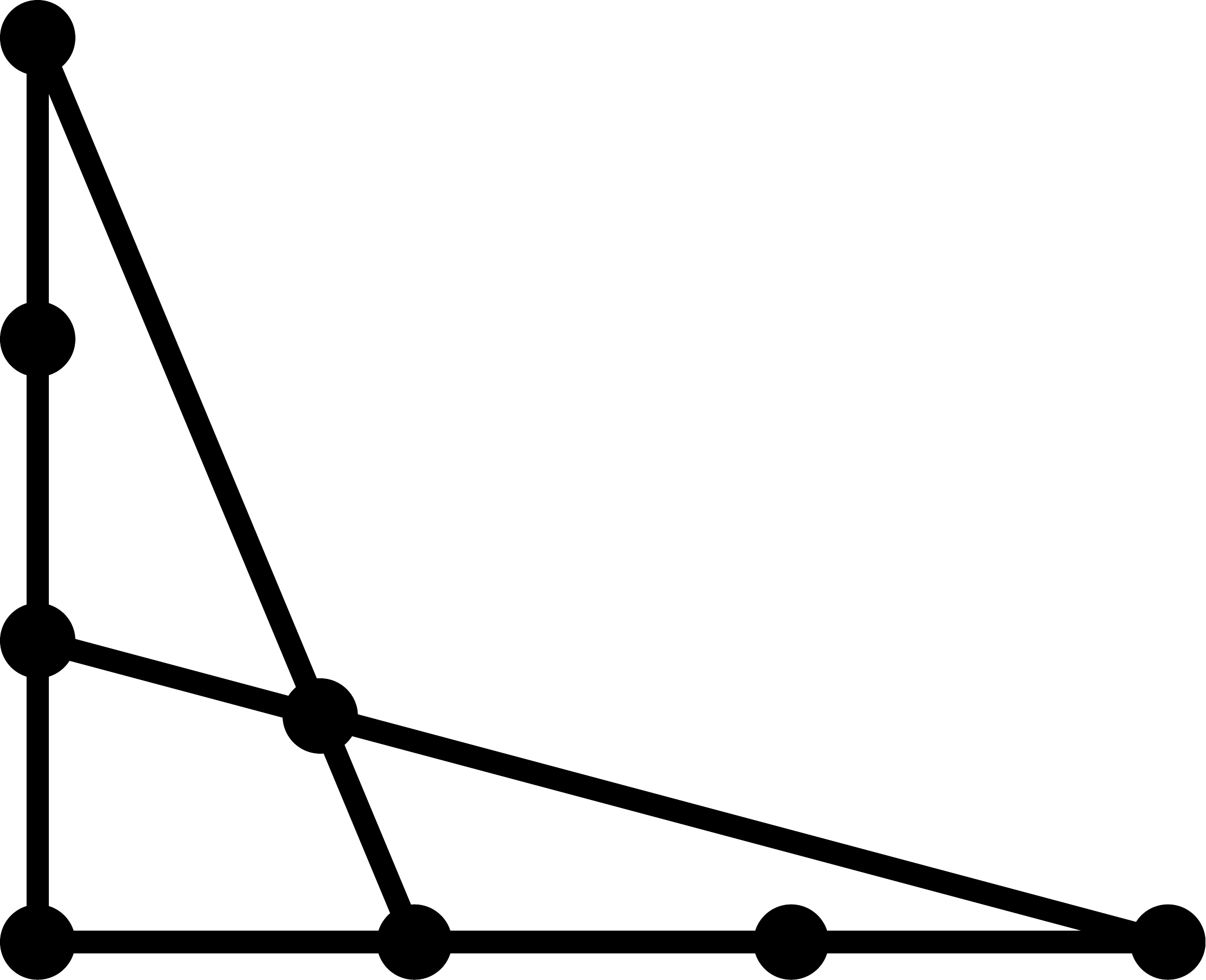}
\caption{}
\label{f:bratios22}
\end{subfigure}
\hspace{0.02\textwidth}
\begin{subfigure}[t]{0.3\textwidth}
\includegraphics[width=1.8in]{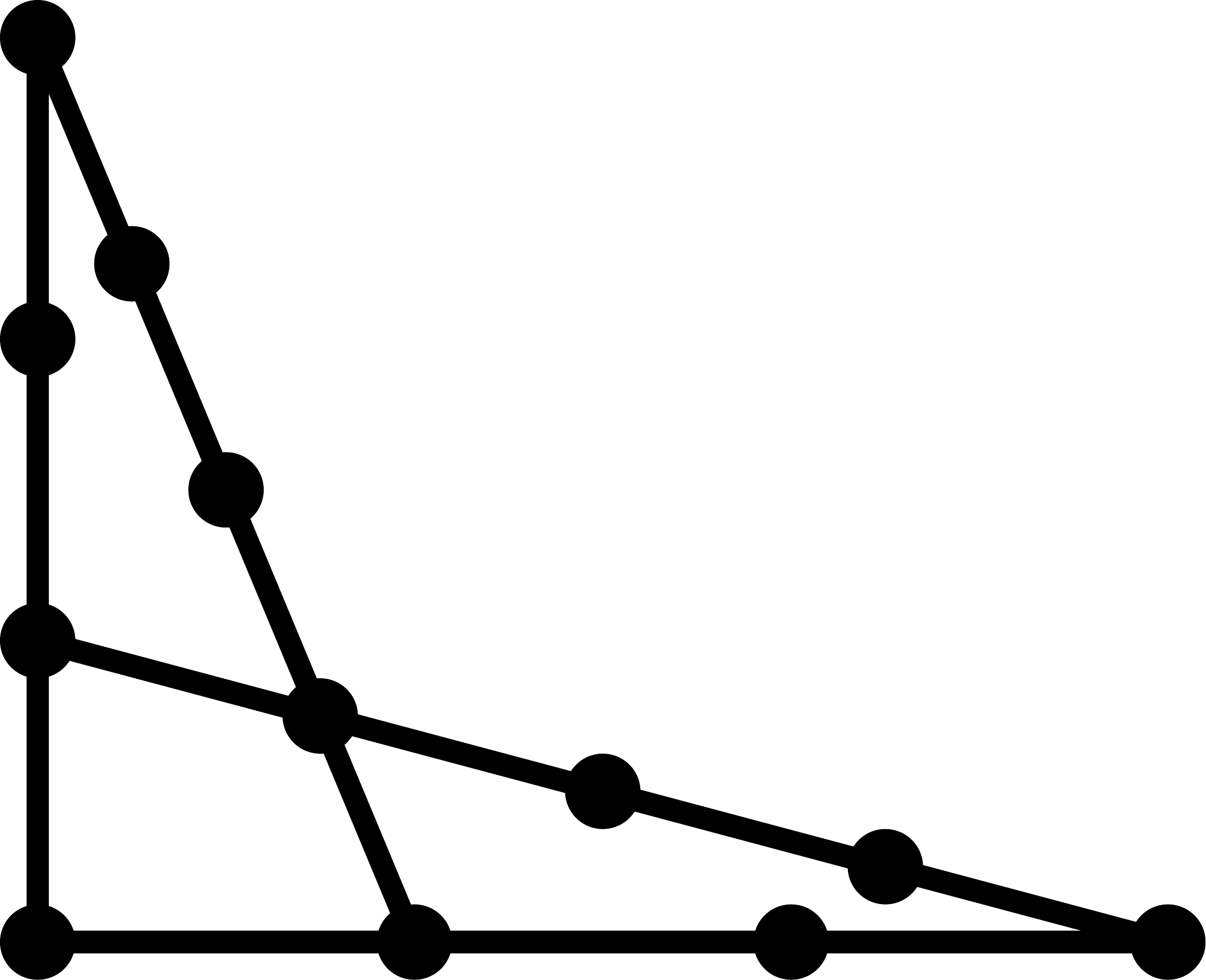}
\caption{}
\label{f:bratios2233}
\end{subfigure}
\hspace{0.02\textwidth}
\begin{subfigure}[t]{0.3\textwidth}
\includegraphics[width=1.8in]{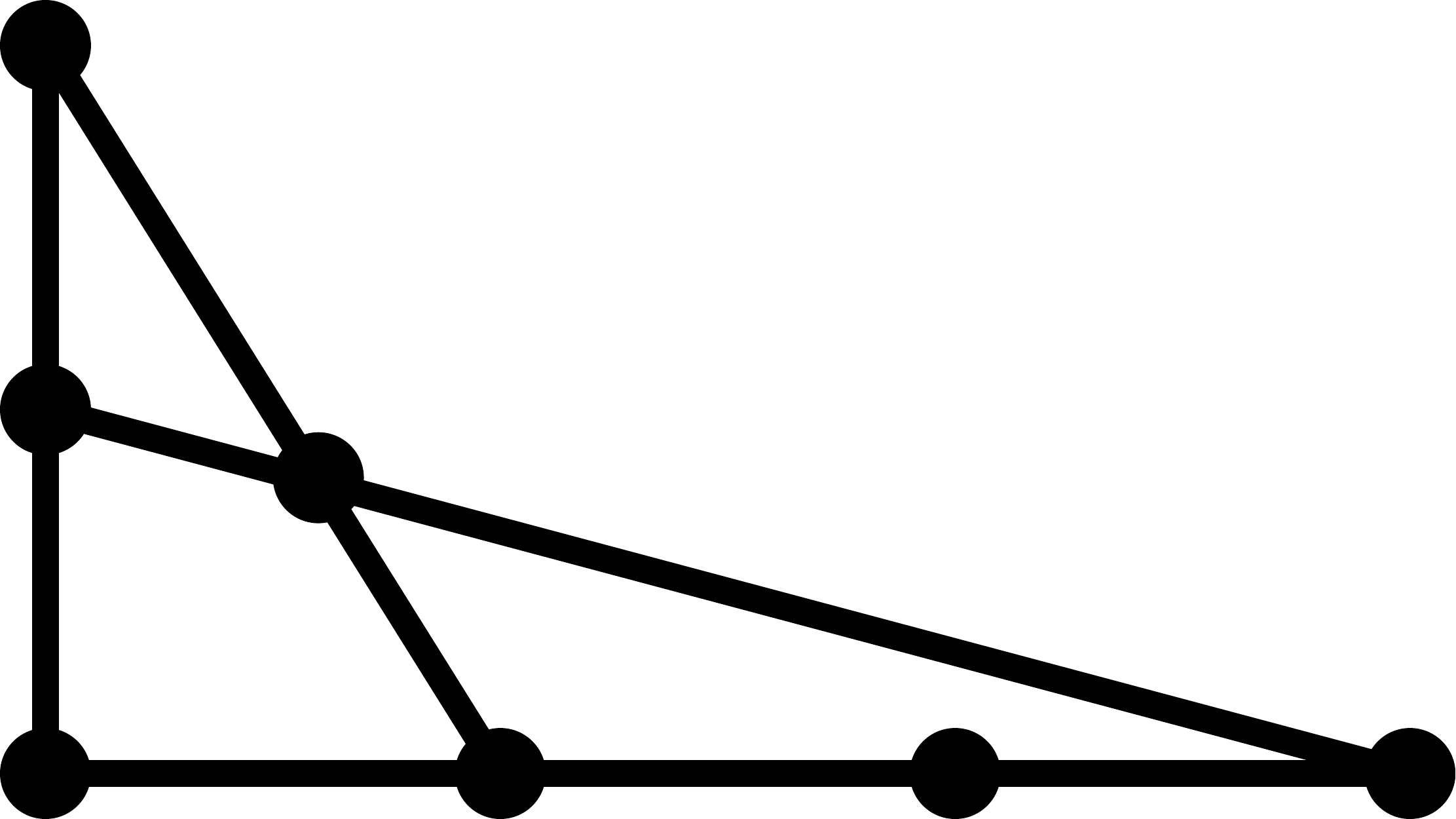}
\caption{}
\label{f:bratios12}
\end{subfigure}
\end{center}
\caption{Configurations in Example~\ref{e:bratios}.}
\end{figure}

\begin{example}\label{e:bratios}
Suppose $r = 5$, and consider the configuration in Figure~\ref{f:bratios22}.  Using the labeling in Lemma~\ref{l:bratios}, $b_1 = b_2 = 2$, which forces $b_3 = b_4 = 3$.  As such, the additional integer points depicted in Figure~\ref{f:bratios2233} must also appear in the configuration.  On the other hand, the configuration in Figure~\ref{f:bratios12} has $b_1 = 1$ and $b_2 = 2$, so by Lemma~\ref{l:bratios}, we have $b_4 = 4$, which is impossible if $r < 6$.  As such, this configuration cannot occur in any 5-segment hypergraph.
\end{example}

We are now ready to prove Theorem \ref{t:r5intersecting}.  

\begin{thm}\label{t:r5intersecting}
If $H$ is an intersecting $5$-segment hypergraph, then $\tau(H) \le 3$.  
\end{thm}

\begin{proof}
Let $H = (V,E)$ be an intersecting 5-segment hypergraph.  We may assume that $|E| \ge 7$, for otherwise the claim $\tau(H) \le 3$ is trivial.  If all of the edges in $H$ intersect in a single vertex, then $\tau(H) = 1$.  Otherwise, choose two edges $e_1, e_2 \in E$, and let $e_3$ be an edge not intersecting $e_1 \cap e_2$. If the three vertices of the triangle formed by $e_1, e_2, e_3$ form a cover of $H$, then again $\tau(H) \le 3$.  Otherwise, there is an edge $e_4$ not containing any of these vertices, so $Q = \{e_1,e_2,e_3,e_4\}$ is a set of four edges satisfying the conditions of Lemma~\ref{l:bratios}.  

Let $b_1, \ldots, b_4$ be the ratios defined in Lemma \ref{l:bratios} for $Q$. 
Since $r = 5$, each $b_i$ must lie in the set $\{\tfrac{1}{3}, \tfrac{1}{2}, 1, 2, 3\}$.  Since the values of $b_1$ and $b_2$ uniquely determine the values of $b_3$ and $b_4$ by Lemma~\ref{l:bratios}, trying each possible pair of values $(b_1, b_2)$ yields 
$$(1,1,2,2), \quad (2,2,3,3), \quad (\tfrac{1}{2},1,1,3), \quad \text{and} \quad (\tfrac{1}{3},\tfrac{1}{2},1,2)$$
as the only possible tuples $(b_1, b_2, b_3, b_4)$ up to symmetry.  We can now divide the proof into four cases, one for each possibility.  We include here only the case where $(b_1, b_2, b_3, b_4) = (\tfrac{1}{3},\tfrac{1}{2},1,2)$, as similar arguments yield the remaining three cases.  

\begin{figure}
\begin{center}
\includegraphics[width=2.8in]{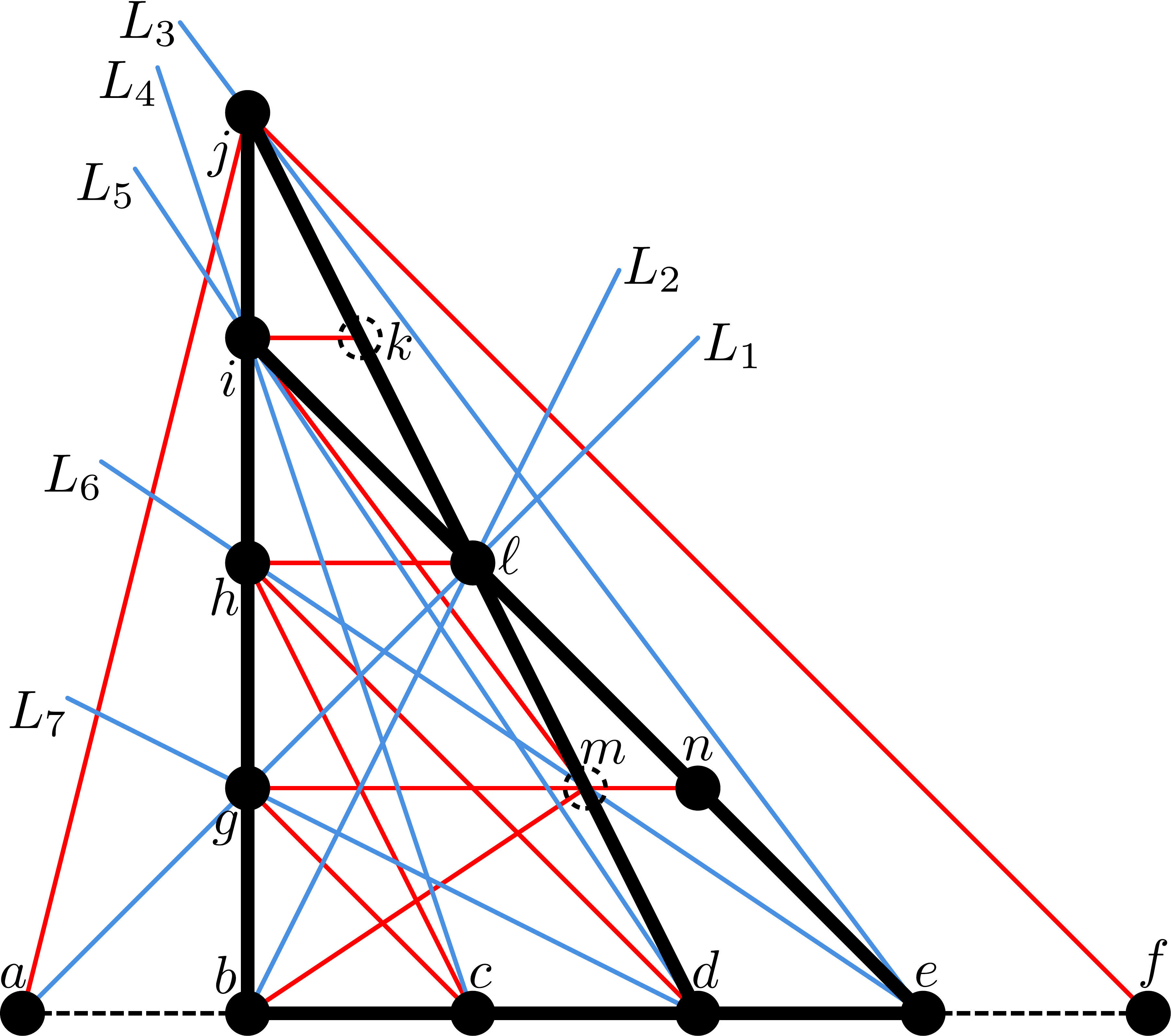}
\hspace{0.1in}
\includegraphics[width=2.8in]{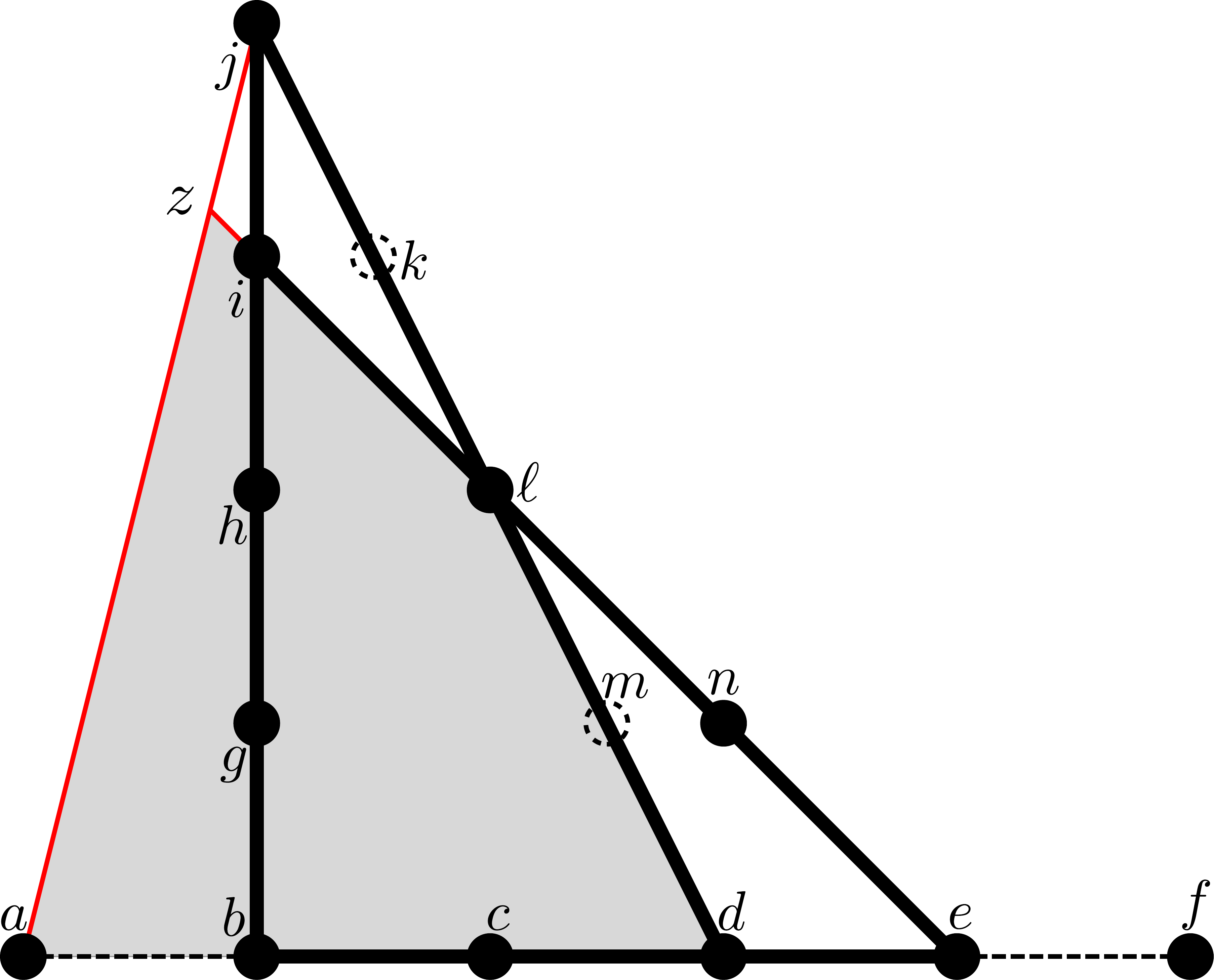}
\end{center}
\caption{
The thick lines (colored black) are the edges in $Q$, with ratios $(b_1, b_2, b_3, b_4) = (\frac{1}{3}, \frac{1}{2}, 1, 2)$.  The thin lines in the left hand image labeled $L_1, \ldots, L_7$ (colored blue) may contain additional edges of $H$, while the remaining thin lines (colored red) cannot contain any edges of~$H$.  Applying Lemma~\ref{l:bratios} to the shaded quadrilateral in the right hand image implies that $\ol{aj}$ cannot contain any edges in $H$.  
}
\label{f:r5intersecting}
\end{figure}

The thick black edges depicted in Figure~\ref{f:r5intersecting} correspond to the edges in $Q$, with corresponding ratios $(b_1, b_2, b_3, b_4) = (\tfrac{1}{3},\tfrac{1}{2},1,2)$.  We first demonstrate that any additional edges in $H$ must lie in one of the lines labeled $L_1, \ldots, L_7$ in Figure~\ref{f:r5intersecting} (i.e.\ those colored blue).  The lines $\ol{gn}$, $\ol{hl}$ and $\ol{ik}$ are parallel to $\ol{be}$ and therefore, since $H$ is intersecting, they cannot contain an edge.  Similarly, the lines $\ol{hd}$, $\ol{gc}$ and $\ol{jf}$ are parallel to  $\ol{ie}$, and $\ol{hc}$ is parallel to $\ol{jd}$, and thus cannot contain an edge.  
Applying Lemma~\ref{l:bratios} to the shaded quadrilateral in Figure~\ref{f:r5intersecting} implies the length ratio of $\ol{jz}$ to $\ol{az}$ is $\frac{1}{4}$, which is impossible if $\ol{aj}$ has only 5 consecutive integer points.  As such, we conclude the line $\ol{aj}$ cannot contain an edge.  By similar arguments, the lines $\ol{im}$ and $\ol{bm}$ cannot contain an edge since they do not intersect $\ol{be}$ and $\ol{ie}$, respectively, at integer points.  

Let $E' = E\setminus Q$.  By the above argument, $|E'| \le 7$ and each edge in $E'$ must lie in a distinct labeled line $L_1, \ldots, L_7$ in Figure~\ref{f:r5intersecting}.  We now consider which pairs of lines can simultaneously contain edges.  For instance, if $L_1$ and $L_3$ both contained edges, then applying Lemma~\ref{l:bratios} to the quadrilateral formed with $\ol{de}$ and $\ol{e\ell}$ implies the edge in $L_1$ contains more than 5 integer points.  Similarly, if $L_1$ and $L_5$ both contained edges, then applying Lemma~\ref{l:bratios} to the quadrilateral formed with $\ol{de}$ and $\ol{e\ell}$ implies $L_1$ contains more than 5 integer points.  Continuing in this way, the only pairs of lines that can contain two distinct edges is $E'$ are the following:
$$\begin{array}{l@{\qquad}l@{\qquad}l@{\qquad}l@{\qquad}l}
\{L_1, L_2\}, &
\{L_1, L_4\}, &
\{L_1, L_7\}, &
\{L_2, L_6\}, &
\{L_3, L_5\}, \\
\{L_3, L_6\}, &
\{L_3, L_7\}, &
\{L_4, L_5\}, &
\{L_5, L_7\}, &
\{L_6, L_7\}.
\end{array}$$
From this, we must have $|E'| \le 3$, as the maximal sets of lines that could simultaneously contain edges in $E'$ are $\{L_3, L_5, L_7\}$ or $\{L_3, L_6, L_7\}$.  However, both triples require $L_3$ to contain more than 5 integer points by Lemma~\ref{l:bratios}.  We conclude that $H$ contains at most 6 edges, which completes the proof of this case.  
\end{proof}

\begin{remark}\label{r:intersectinglinecount}
One consequence of the proof of Theorem~\ref{t:r5intersecting} is that any intersecting 5-segmant hypergraph has at most 6 edges if it contains a triangle.  Indeed, in any such configuration, the degree of each vertex $v$ is at most $r$ since every edge through $v$ must intersect a edge not containing $v$ at a different vertex.  As such, any configuration with more than $r+1$ edges contains a configuration of the form in Lemma~\ref{l:bratios}, and the proof of Theorem~\ref{t:r5intersecting} verifies that no more than 2 edges can be added in each case.  One configuration achieving the maximum 6 edges is illustrated in Figure~\ref{f:r5intersectingmonsters}.  
\end{remark}

\begin{question}\label{q:intersectinglinecount}
What is the maximum number of edges that can occur in an intersecting $r$-segment hypergraph containing a triangle?  
\end{question}

%%%%%%%%%%%%%%%%%%%%%%%%%%%%%%%%%%%%%%%%%%%%%%%%%%%%%%%%%%%%%%%%%%%%%%%%%
\section{Fractional covers and matchings in $r$-segment hypergraphs}%%%%%
\label{sec:fractional}%%%%%%%%%%%%%%%%%%%%%%%%%%%%%%%%%%%%%%%%%%%%%%%%%%%
%raggedbottom%%%%%%%%%%%%%%%%%%%%%%%%%%%%%%%%%%%%%%%%%%%%%%%%%%%%%%%%%%%%

Let $H = (V,E)$ be a hypergraph.  A function $f:E \to \R_{\ge 0}$ is a \emph{fractional matching} if $\sum_{e \ni v} f(e) \le 1$ for every $v \in V$.  We denote by 
$$\nu^*(H) = \max \bigg\{\!\sum_{e \in E} f(e) : f \text{ is a fractional matching of } H\!\bigg\}$$
the \emph{fractional matching number} of $H$.  Similarly, a \emph{fractional cover } is a function $g:V \to \R_{\ge 0}$ such that $\sum_{v \in e} g(v) \ge 1$ for every $e \in E$.  We denote by 
$$\tau^*(H) = \min \bigg\{\!\sum_{v \in V}^{\vspace{-0.01in}} g(v) : g \text{ is a fractional cover of } H\!\bigg\}$$
the \emph{fractional covering number} of $H$.  In every hypergraph $H$, we have
$$\nu(H) \le \nu^*(H) = \tau^*(H) \le \tau(H),$$
where the equality $\nu^*(H) = \tau^*(H)$ follows from linear programming duality.  Moreover, the \emph{complementary slackness condition} \cite{Krivelevich} asserts that if some minimum factional cover $g:V \to \R_{\ge 0}$ satisfies $g(u) > 0$ for some $u \in V$, then every maximum fractional matching $f:E \to \R_{\ge 0}$ satisfies $\sum_{e \ni u} f(e) = 1$.  

\begin{thm}\label{t:fractional}
If $H$ is an $r$-segment hypergraph with $r \ge 3$, then $\tau^*(H) \le (r-1)\nu(H)$ and $\tau(H) \le (r-1)\nu^*(H)$.    
\end{thm}

In the proof of Theorem \ref{t:fractional}, we will use a theorem of F\"uredi.  

\begin{thm}[{\cite{furedi}}]\label{t:furedi}
If every edge of a hypergraph $H$ has at most $r \ge 3$ vertices, and $H$ does not contain a copy of the $r$-uniform projective plane, then
$\nu(H) \ge \frac{\nu^*(H)}{r-1}.$
\end{thm}

\begin{proof}[Proof of Theorem \ref{t:fractional}]
By Theorem~\ref{t:furedi}, for the first statement of the theorem it is enough to show that $H$ does not contain a copy of the $r$-uniform projective plane.  Suppose to the contrary that $H$ contains a subhypergraph $H'$ isomorphic to the $r$-uniform projective plane. Since $H'$ is an intersecting $r$-segment hypergraph, Theorem~\ref{t:isolatedvertex} implies $H'$ contains an isolated vertex, contradicting the fact that each vertex in the $r$-uniform projective plane belongs to $r$ edges. 

Suppose $H = (V,E)$ has the minimal number of edges such that 
$$\tau(H) > (r-1)\nu^*(H) = (r-1)\tau^*(H).$$
Let $f:E \to \RR_{\ge 0}$ and $g:V \to \RR_{\ge 0}$ be a maximal fractional matching and a minimal fractional cover of $H$, respectively. By removing vertices from $V$ if necessary, we may assume that every vertex in $V$ belongs to some edge. 

Suppose first that some vertex $u \in V$ has $g(u) = 0$.  Given an edge $e$ containing~$u$, we have $\sum_{e \ni v} g(v) \ge 1$ since $g$ is the fractional cover, but $g(u) = 0$, so there must be a vertex $v \in e$ with $g(v) \ge \frac{1}{r-1}$.  
Consider the hypergraph $H'$ obtained from $H$ by removing $v$ and all the edges containing it. 

Clearly $\tau(H') \ge \tau(H)-1$, since otherwise a minimal cover in $H'$ together with $v$ is a cover of $H$ of size smaller than $\tau(H)$.  Moreover, by the minimality of $H$, we have $\tau(H') \le (r-1)\tau^*(H')$, and since the restriction of $g$ to $V(H')$ is a fractional cover of~$H'$, we have $\tau^*(H') \le \tau^*(H) - g(v)$.  Thus, 
$$\tau^*(H') \le \tau^*(H) - g(v) < \frac{\tau(H)}{r-1} - \frac{1}{r-1} = \frac{\tau(H)-1}{r-1} \le \frac{\tau(H')}{r-1} \le \tau^*(H'),$$
which is a contradiction.
 
We may now assume $g(u)>0$ for all $u \in V$ (if the second inequality does not hold for some $u\in V$  then we can repeat the argument above by removing $u$ from $H$).  By the complementary slackness conditions, each vertex $u$ satisfies $\sum_{e \ni u} f(e) = 1$, so
$$|V| = \sum_{u \in V} 1 =  \sum_{u \in V}\sum_{e \ni u} f(e) = \sum_{e \in E} rf(e) = r\nu^*(H) =  r\tau^*(H).$$
We conclude $\tau^*(H) = \frac{1}{r}|V|$, but since $\tau(H)\le \frac{2}{3}|V|$ by Corollary \ref{c:weakcoloring}, we have 
$$\frac{\tau(H)}{\tau^*(H)} \le \frac{\frac{2}{3}|V|}{\frac{1}{r}|V|} = \frac{2r}{3} \le r-1,$$
for every $r \ge 3$. This concludes the proof of the theorem.
\end{proof}

\subsection*{Acknowledgements}
The first author was supported by PASPA (UNAM) and CONACYT during her sabbatical visit, as well as by Proyecto PAPIIT 104915, 106318 and CONACYT Ciencia B\'asica 282280.   The first author would also like to express her appreciation of the UC Davis Mathematics department's hospitality during her visit.

%%%%%%%%%%%%%%%%%%%%%%%%%%%%%%%%%%%%%%%%%%%%%%%%%%%%%%%%%%%%%%%%%%%%%%%%%
%%%%%%%%%%%%%%%%%%%%%%%%%%%%%%%%%%%%%%%%%%%%%%%%%%%%
%%%%%%%%%%%%%%%%%%%%%%%%%%%%%%%%%%%%%%%%%%%%%%%%%%%%%%%%%%%%%%%%%%%%%%%%%

\end{document}